\documentclass{amsart}
\usepackage{amssymb,cmap,faktor}
\newtheorem*{thma}{Theorem~A}
\newtheorem*{thmb}{Theorem~B}
\newtheorem*{thmc}{Theorem~C}
\newtheorem*{thmd}{Theorem~D}
\newtheorem{thm}{Theorem}[section]
\newtheorem{cor}[thm]{Corollary}
\newtheorem{prop}[thm]{Proposition}
\newtheorem{fact}[thm]{Fact}
\newtheorem{lemma}[thm]{Lemma}
\newtheorem{claim}{Claim}[thm]

\theoremstyle{definition}
\newtheorem{defn}[thm]{Definition}

\newtheorem{definition}[thm]{Definition}
\newtheorem{q}[thm]{Question}
\newtheorem{conv}[thm]{Convention}

\theoremstyle{remark}
\newtheorem{remark}[thm]{Remark}
\newtheorem{remarks}{Remark}

\newcommand*\axiomfont[1]{\textsf{\textup{#1}}}

\newcommand\zfc{\axiomfont{ZFC}}

\newcommand\s{\subseteq}
\newcommand\sq{\sqsubseteq}
\newcommand\symdiff{\triangle}
\newcommand\conc{{^\smallfrown}}
\newcommand\last[2]{\eth_{#1,#2}}
\newcommand\br{\blacktriangleright}
\renewcommand\mid{\mathrel{|}\allowbreak}
\newcommand\Mid{\mathrel{}\middle|\mathrel{}}
\renewcommand{\restriction}{\mathbin\upharpoonright}
\newcommand{\stick}{{{\ensuremath \mspace{2mu}\mid\mspace{-12mu} {\raise0.6em\hbox{$\bullet$}}}}}
\newcommand{\pro}[5]{\pr_1(#1,\allowbreak\faktor{{\scriptstyle{#2\circledast#1}}}{ {}^{#3\circledast#1}},\allowbreak#4,#5)}

\DeclareMathOperator{\reg}{Reg}
\DeclareMathOperator{\cf}{cf}
\DeclareMathOperator{\cl}{cl}
\DeclareMathOperator{\Tr}{Tr}
\DeclareMathOperator{\tr}{tr}
\DeclareMathOperator{\im}{Im}
\DeclareMathOperator{\otp}{otp}
\DeclareMathOperator{\dom}{dom}
\DeclareMathOperator{\add}{Add}
\DeclareMathOperator{\acc}{acc}
\DeclareMathOperator{\nacc}{nacc}
\DeclareMathOperator{\p}{P}
\DeclareMathOperator{\pr}{Pr}
\DeclareMathOperator{\pl}{P\ell}
\DeclareMathOperator{\ssup}{ssup}

\author{Assaf Rinot}
\address{Department of Mathematics, Bar-Ilan University, Ramat-Gan 5290002, Israel.}
\urladdr{http://www.assafrinot.com}

\author{Jing Zhang}
\address{Department of Mathematics, Bar-Ilan University, Ramat-Gan 5290002, Israel.}
\urladdr{https://jingjzzhang.github.io/}

\title{Strongest transformations}
\begin{document}
\begin{abstract} We continue our study of maps transforming high-dimensional complicated objects into squares of stationary sets.
Previously, we proved that many such transformations exist in $\zfc$,
and here we address the consistency of the strongest conceivable transformations.

Along the way, we obtain new results on Shelah's coloring principle $\pr_1$.
For $\kappa$ inaccessible, we prove the consistency of $\pr_1(\kappa,\kappa,\kappa,\kappa)$.
For successors of regulars, we obtain a full lifting of Galvin's 1980 theorem.
In contrast, the full lifting of Galvin's theorem to successors of singulars is shown to be inconsistent.
\end{abstract}
\date{Preprint as of April~30, 2021. For the latest version, visit \textsf{http://p.assafrinot.com/45}.}
\maketitle
\section{Introduction}
Throughout the paper, $\kappa$ denotes a regular uncountable cardinal,
and $\theta,\chi$ denote (possibly finite) cardinals $\le\kappa$. 
In \cite{paper44}, the authors introduced the transformation principle $\pl_1(\kappa,\theta,\chi)$ (see Definition~\ref{fulldefpl1}),
proved that it is strictly stronger than Shelah's coloring principle $\pr_1(\kappa,\kappa,\theta,\chi)$,
proved that $\pl_1(\kappa,1,\chi)$ implies that $\pr_1(\kappa,\kappa,\theta,\chi)$
is no stronger than the classical negative partition relation $\kappa\nrightarrow[\kappa]^2_\theta$,
and demonstrated the utility of this reduction in additive Ramsey theory.

Most of \cite{paper44} was devoted to providing sufficient conditions for instances of $\pl_1(\ldots)$ to hold.
In particular, combining walks on ordinals with strong forms of the oscillation oracle $\pl_6(\ldots)$, 
it was shown that  many instances of $\pl_1(\ldots)$  are theorems of $\zfc$.
Note that even the very weak instance $\pl_1(\kappa,1,3)$ is quite powerful,
as it allows the transformation of rectangles into squares, as in \cite{paper13}.

The current paper is dedicated to studying the strongest instance of $\pl_1(\kappa,\theta,\chi)$,
being $\theta:=\kappa$ and $\chi:=\sup(\reg(\kappa))$ and a further strengthening of which, as follows:

\begin{defn} For a stationary subset $\Gamma\s\kappa$,
$\pl_2(\kappa,\Gamma,\chi)$ asserts the existence of a transformation $\mathbf t:[\kappa]^2\rightarrow[\kappa]^2$ satisfying the following:
\begin{itemize}
\item for every $(\alpha,\beta)\in[\kappa]^2$, if $\mathbf t(\alpha,\beta)=(\alpha^*,\beta^*)$, then $\alpha^*\le\alpha<\beta^*\le\beta$;
\item for every $\sigma<\chi$ and every pairwise disjoint subfamily $\mathcal A\s[\kappa]^{\sigma}$ of size $\kappa$,
there exists a club $D\s\kappa$, such that, for all $(\alpha^*,\beta^*)\in[\Gamma\cap D]^2$,
there exists $(a,b)\in[\mathcal A]^2$ with $\mathbf t[a\times b]=\{(\alpha^*,\beta^*)\}$.
\end{itemize}
\end{defn}
Our first main result concerns successors of regular cardinals:

\begin{thma} For every infinite regular cardinal $\mu$, any of the following imply that
$\pl_2(\mu^+,E^{\mu^+}_\mu,\mu)$ holds:
\begin{enumerate}
\item $\stick(\mu^+)$;
\item  $(\mu^+)^{\aleph_0}=\mu^+$ and $\clubsuit(S)$ holds for some nonreflecting stationary $S\s\mu^+$.
\end{enumerate}
\end{thma}

\begin{remarks} $\stick(\mu^+)$ is the \emph{stick} principle (see Definition~\ref{stickp}) which is a weakening of the assertion that $2^\mu=\mu^+$.
In Clause~(2), we moreover get  $\pl_2(\mu^+,\mu^+,\mu)$.
\end{remarks}
Theorem~A sheds a new light on \cite[Question~2.3]{Sh:572},
in particular showing that, for every infinite regular cardinal $\mu$, $2^\mu=\mu^+$ implies $\pr_1(\mu^+,\mu^+,\mu^+,\mu)$.
The case $\mu=\aleph_0$ was proved by Galvin back in 1980 \cite{galvin},
but his original proof only generalizes to show that $2^\mu=\mu^+$ implies $\pr_1(\mu^+,\mu^+,\mu^+,\aleph_0)$.

\medskip

Our second main result concerns (weakly) inaccessible cardinals:

\begin{thmb} For every inaccessible cardinal $\kappa$, any of the following imply that
$\pl_1(\kappa,\kappa,\kappa)$ holds:
\begin{enumerate}
\item $\square(\kappa)$ and $\diamondsuit(S)$ for some stationary $S\s\kappa$ that does not reflect at regulars;
\item $\square(\kappa)$ and $\diamondsuit^*(\kappa)$;
\item $\boxtimes^-(\kappa)$ and $\diamondsuit(\kappa)$;
\item $\kappa$ is Mahlo, $\diamondsuit(S)$ for some stationary $S\s\kappa$ that does not reflect,
and there exists a nonreflecting stationary subset of $\reg(\kappa)$. 
\end{enumerate}
\end{thmb}
\begin{remarks} $\boxtimes^-(\kappa)$  is a simple instance of the Brodsky-Rinot proxy principle (see Definition~\ref{xbox}).  In Clauses (3) and (4), we moreover get  $\pl_2(\kappa,\kappa,\kappa)$.
\end{remarks}
 
Our third main result concerns successors of singulars. Here, we 
uncover $\zfc$ constraints on the extent of the combinatorial principles under discussion.
\begin{thmc} Let $\mu$ be a singular cardinal. Then:
\begin{enumerate}
\item  $\pr_1(\mu^+,\mu^+,2,\mu)$ fails;
\item $\pr_1(\mu^+, \mu^+,2, \cf(\mu)^+)$ fails, provided that 
 $\mu$ is a (singular) limit of strongly compact cardinals;
\item $\pr_6(\mu^+,\mu^+,2,\mu)$ fails;
\item $\pl_6(\mu^+,\mu)$ fails.
\end{enumerate}
\end{thmc}

As a corollary, we confirm that in G\"odel's constructible universe,
every regular uncountable cardinal admits the strongest conceivable transformation:
\begin{thmd} Assuming $V=L$, for every regular uncountable cardinal $\kappa$ and every regular cardinal $\chi\le\chi(\kappa)$,
$\pl_2(\kappa,\kappa,\chi)$ holds.
\end{thmd}

\subsection{Organization of this paper}
In Section~\ref{subsectionwalks}, we define all the combinatorial principles discussed in this paper,
prove Theorem~C and also prove 
that $\pl_6(\mu,\mu)$ and $\pr_6(\mu,\mu,2,\mu)$ fail for any infinite cardinal $\mu$.
These theorems should be understood as verifying positive partition relations.

In Section~\ref{cseqsection}, we make some contributions to the theory of $C$-sequences. 
This will play a role in the proofs of Section~\ref{pumpup}.

In Section~\ref{pumpup}, we provide sufficient conditions for $\pr_1(\kappa,\kappa,\kappa,\chi)$ to imply $\pl_1(\kappa,\kappa,\chi)$
or $\pl_2(\kappa,\Gamma,\chi)$.

In Section~\ref{inaccessibles}, we deal with inaccessible cardinals, in particular,
proving Clauses (1)--(3) of Theorem~B.

In Section~\ref{galvin}, we first prove our generalization of Galvin's theorem
and then use the results of Section~\ref{pumpup} to obtain Clause~(1) of Theorem~A.

In Section~\ref{proxysect}, we derive transformations by walking along $C$-sequences witnessing an instance of the Brodsky-Rinot proxy principle. 
This is how Clause~(2) of Theorem~A, Clause~(4) of Theorem~B,
and Theorem~D are obtained.

\subsection{Notation and conventions}
By an \emph{inaccessible} we mean a regular uncountable limit cardinal.
Let $E^\kappa_\chi:=\{\alpha < \kappa \mid \cf(\alpha) = \chi\}$,
and define $E^\kappa_{\le \chi}$, $E^\kappa_{<\chi}$, $E^\kappa_{\ge \chi}$, $E^\kappa_{>\chi}$,  $E^\kappa_{\neq\chi}$ analogously.
The collection of all sets of hereditary cardinality less than $\kappa$ is denoted by $\mathcal H_\kappa$.
The set of all infinite and regular cardinals below $\kappa$ is denoted by $\reg(\kappa)$.
The length of a finite sequence $\varrho$ is denoted by $\ell(\varrho)$.
A stationary subset $S\s\kappa$ is \emph{nonreflecting} iff there exists no $\alpha\in E^\kappa_{>\omega}$ such that $S\cap\alpha$ is stationary in $\alpha$.
For a set of ordinals $a$, we write 
$\ssup(a):=\sup\{\alpha+1\mid \alpha\in a\}$,
$\acc^+(a) := \{\alpha < \ssup(a) \mid \sup(a \cap \alpha) = \alpha > 0\}$,
$\acc(a) := a \cap \acc^+(a)$, $\nacc(a) := a \setminus \acc(a)$,
and $\cl(a):= a\cup\acc^+(a)$.
For sets of ordinals, $a$ and $b$, we let $a\circledast b:=\{(\alpha,\beta)\in a\times b\mid \alpha<\beta\}$,
and write $a < b$ to express that $a\times b$ coincides with $a\circledast b$.
For any set $\mathcal A$, we write
$[\mathcal A]^\chi:=\{ \mathcal B\s\mathcal A\mid |\mathcal B|=\chi\}$ and
$[\mathcal A]^{<\chi}:=\{\mathcal B\s\mathcal A\mid |\mathcal B|<\chi\}$.
This convention admits two refined exceptions:
\begin{itemize}
\item For an ordinal $\sigma$ and a set of ordinals $A$, we write 
$[A]^\sigma$ for $\{ B\s A\mid \otp(B)=\sigma\}$;
\item For a set $\mathcal{A}$ which is either an ordinal or a collection of sets
of ordinals, we interpret $[\mathcal{A}]^2$ as the collection of \emph{ordered} pairs $\{ (a,b)\in\mathcal A\times\mathcal A\mid a<b\}$.
\end{itemize}
In particular, $[\kappa]^2=\{(\alpha,\beta)\mid \alpha<\beta<\kappa\}$.
Likewise, we let $[\kappa]^3:=\{(\alpha,\beta,\gamma)\in\kappa\times\kappa\times\kappa\mid \alpha<\beta<\gamma<\kappa\}$.

\section{Combinatorial principles and inconsistent instances of which}\label{subsectionwalks}\label{positive}

\begin{conv}\label{conv28} For any coloring $f:[\kappa]^2\rightarrow\theta$ and $\delta<\kappa$,
while $(\delta,\delta)\notin[\kappa]^2$, we extend the definition of $f$, and agree to let $f(\delta,\delta):=0$.
\end{conv}

\begin{defn}[\cite{shelah_productivity}]\label{def_pr1}
$\pr_1(\kappa, \kappa, \theta, \chi)$ asserts the existence of a coloring $c:[\kappa]^2 \rightarrow \theta$ such that for every  $\sigma<\chi$, every  
pairwise disjoint subfamily $\mathcal{A}  \subseteq [\kappa]^{\sigma}$ of size $\kappa$,
and every $\tau< \theta$, there is $(a,b) \in [\mathcal{A}]^2$ such that $c[a \times b] = \{\tau\}$.
\end{defn}

\begin{defn}[\cite{paper49}]
$\pro{\kappa}{\mu}{1}{\theta}{\chi}$ asserts the existence of a coloring $c:[\kappa]^2\rightarrow\theta$
satisfying that for every $\sigma<\chi$, and every pairwise disjoint subfamilies $\mathcal A,\mathcal B$ of $[\kappa]^\sigma$ with $|\mathcal A|=\mu$ and $|\mathcal B|=\kappa$,
there is $a\in\mathcal A$ such that,
for every $\tau<\theta$,
there is $b\in\mathcal B$ with $a<b$ such that $c[a\times b]=\{\tau\}$.
\end{defn}

\begin{defn}[\cite{paper44}]\label{fulldefpl1} $\pl_1(\kappa,\theta,\chi)$ asserts the existence of a transformation $\mathbf t:[\kappa]^2\rightarrow[\kappa]^3$ satisfying the following:
\begin{itemize}
\item for every $(\alpha,\beta)\in[\kappa]^2$, if $\mathbf t(\alpha,\beta)=(\tau^*,\alpha^*,\beta^*)$, then $\tau^*\le\alpha^*\le\alpha<\beta^*\le\beta$;
\item for every $\sigma<\chi$ and every pairwise disjoint subfamily $\mathcal A\s[\kappa]^{\sigma}$ of size $\kappa$,
there exists a stationary $S\s\kappa$ such that, for all $(\alpha^*,\beta^*)\in[S]^2$ and $\tau^*<\min\{\theta,\alpha^*\}$,
there exist $(a,b)\in[\mathcal A]^2$ with $\mathbf t[a\times b]=\{(\tau^*,\alpha^*,\beta^*)\}$.
\end{itemize}
\end{defn}

Note that for every stationary $\Gamma\s\kappa$, $\pl_2(\kappa,\Gamma,\chi)$ implies $\pl_1(\kappa,\kappa,\chi)$.
By \cite[\S2.3]{paper44}, $\pl_1(\kappa,\theta,\chi)$ implies $\pr_1(\kappa,\kappa,\theta+\aleph_0,\chi)$.

\begin{defn}[\cite{Sh:572}] $\pr_6(\kappa,\kappa,\theta,\chi)$ asserts the existence of a coloring $d:{}^{<\omega}\kappa\rightarrow\theta$
satisfying the following. 
For every $\tau<\theta$,
and every sequence $\langle (u_\alpha,v_\alpha)\mid\alpha\in E\rangle$ such that:
\begin{enumerate}
\item $E$ is a club in $\kappa$;
\item  $u_\alpha$ and $v_\alpha$ are nonempty elements of $[{}^{<\omega}\kappa]^{<\chi}$;
\item $\alpha\in\im(\varrho)$ for all $\varrho\in u_\alpha$;
\item $\alpha\in\im(\sigma)$ for all $\sigma\in v_\alpha$,
\end{enumerate}
there exists $(\alpha,\beta)\in[E]^2$ such that $d(\varrho{}^\frown\sigma)=\tau$ for all $\varrho\in u_\alpha$ and $\sigma\in v_\beta$.
\end{defn}

\begin{defn}[\cite{paper15}]\label{defpl6} $\pl_6(\kappa,\chi)$ asserts the existence of a map $d:{}^{<\omega}\kappa\rightarrow\omega$
satisfying the following. For every sequence $\langle (u_\alpha,v_\alpha,\sigma_\alpha)\mid\alpha<\kappa\rangle$ and $\varphi:\kappa\rightarrow\kappa$ with
\begin{enumerate}
\item $\varphi$ is eventually regressive. That is, $\varphi(\alpha)<\alpha$ for co-boundedly many $\alpha<\kappa$;
\item  $u_\alpha$ and $v_\alpha$ are nonempty elements of $[{}^{<\omega}\kappa]^{<\chi}$;
\item $\alpha\in\im(\varrho)$ for all $\varrho\in u_\alpha$;
\item $\sigma_\alpha{}^\frown\langle\alpha\rangle\sqsubseteq \sigma$ for all $\sigma\in v_\alpha$,
\end{enumerate}
there exists $(\alpha,\beta)\in[\kappa]^2$ with $\varphi(\alpha)=\varphi(\beta)$ such that $d(\varrho\conc\sigma)=\ell(\varrho)$ for all $\varrho\in u_\alpha$ and $\sigma\in v_\beta$.
\end{defn}

\begin{prop} Let $\mu$ be a singular cardinal. 

Then $\pr_1(\mu^+,\mu^+,2,\mu)$ and $\pl_1(\mu^+,1,\mu)$ fail.
\end{prop}
\begin{proof} As mentioned earlier, $\pl_1(\mu^+,1,\mu)$ implies $\pr_1(\mu^+, \mu^+,\aleph_0, \mu)$,
so, towards a contradiction, let us suppose that $c:[\mu^+]^2\rightarrow2$ witnesses $\pr_1(\mu^+,\mu^+,\allowbreak2,\mu)$.

\begin{claim} There exists $A\in[\mu^+]^{\mu^+}$ such that, for every $a\in[A]^{<\mu}$,
there are cofinally many $\beta<\mu^+$ with $c[a\times\{\beta\}]=\{0\}$.
\end{claim}
\begin{proof} Suppose not. Construct a sequence $\langle (A_i,a_i,B_i)\mid i<\mu^+\rangle$ by recursion on $i<\mu^+$, as follows:

$\br$ Let $A_0:=\mu^+$. By the indirect assumption, we may find $a_0\in[\mu^+]^{<\mu}$
such that $B_0:=\{\beta<\mu^+\mid \beta>\sup(a_0)\ \&\ c[a_0\times\{\beta\}]=\{0\}\}$ is bounded in $\mu^+$.

$\br$ Suppose that $i<\mu^+$ is nonzero, and that $\langle (A_j,a_j,B_j)\mid j<i\rangle$ has already been defined.
Set $A_i:=\mu^+\setminus(\sup(\bigcup_{j<i}B_j)+1)$.
By the indirect assumption, we may now find $a_i\in[A_i]^{<\mu}$
such that $B_i:=\{\beta<\mu^+\mid \beta>\sup(a_i)\ \&\ c[a_i\times\{\beta\}]=\{0\}\}$ is bounded in $\mu^+$.

This completes the description of the recursion.
Fix $\sigma<\mu$ and $I\in[\mu^+]^{\mu^+}$ such that $\otp(a_i)=\sigma$ for all $i\in I$.
Then $\langle a_i\mid i\in I\rangle$ is a $<$-increasing sequence of elements of $[\mu^+]^\sigma$.
Thus, by the choice of $c$, we may find $(j,i)\in[I]^2$ such that $c[a_j\times a_i]=\{0\}$.
However, $a_i\s A_i$, so that $a_i\cap B_j=\emptyset$. This is a contradiction.
\end{proof}
Fix $A$ as in the claim. Without loss of generality, $\min(A)\ge\mu$.
Let $\delta\in A$. Fix a decomposition $A\cap \delta=\biguplus_{i<\cf(\mu)}A_{\delta,i}$ such that $|A_{\delta,i}|<\mu$ for all $i<\cf(\mu)$,
and then, for every $i<\cf(\mu)$, fix $\beta_{\delta,i}>\delta$ such that $d[A_{\delta,i}\times\{\beta_{\delta,i}\}]\s\{0\}$.
Denote $b_\delta:=\{ \delta, \beta_{\delta,i}\mid i<\cf(\mu)\}$,
so that $b_\delta\in[\mu^+\setminus\delta]^{\le\cf(\mu)}$.

Fix a sparse enough $\Delta\in[A]^{\mu^+}$ such that, for all $(\gamma,\delta)\in[\Delta]^2$,
$\sup(b_\gamma)<\min(b_\delta)$.
Now, by the choice of $c$, there must exist $(\gamma,\delta)\in[\Delta]^2$ such that $c[b_\gamma\times b_\delta]=\{1\}$.
Pick $i<\cf(\mu)$ such that $\gamma\in A_{\delta,i}$.
Then $d(\gamma,\beta_{\delta,i})=0$,
contradicting the fact that $(\gamma,\beta_{\delta,i})\in b_\gamma\times b_\delta$.
\end{proof}

The preceding proof makes it clear that the following holds, as well.
\begin{prop} For every infinite regular cardinal $\mu$,
$\pr_1(\mu^+,\mu^+,2,\mu^+)$ fails.\qed
\end{prop}

The next result is suggested by the proof of \cite[Theorem~2.14]{paper34}.

\begin{prop}  Suppose that $\mu$ is a singular limit of strongly compact cardinals.

Then $\pr_1(\mu^+, \mu^+,2, \cf(\mu)^+)$ and $\pl_1(\mu^+,1,\cf(\mu)^+)$ fail. 
\end{prop}
\begin{proof} As mentioned earlier, $\pl_1(\mu^+,1,\cf(\mu)^+)$ implies $\pr_1(\mu^+, \mu^+,\aleph_0, \cf(\mu)^+)$,
so, towards a contradiction, let us suppose that $c:[\mu^+]^2\rightarrow2$ witnesses $\pr_1(\mu^+,\mu^+,\allowbreak2,\cf(\mu)^+)$.

\begin{claim} Let $\theta<\mu$. There exists $X\in[\mu^+]^{\mu^+}$ and $j<2$ such that, for every $x\in[X]^{\theta}$,
there are cofinally many $\beta\in X$ with $c[x\times\{\beta\}]=\{j\}$.
\end{claim}
\begin{proof}
For all $\alpha<\mu^+$ and $j<2$, let $$B^\alpha_j:=\{ \beta<\mu^+\mid \beta>\alpha\ \&\ c(\alpha,\beta)=j\}.$$
As there exists a strongly compact cardinal in-between $\theta$ and $\mu$,
let us fix a uniform, $\theta^+$-complete ultrafilter $U$ on $\mu^+$.
Then, for each $\alpha < \mu^+$, find $j_{\alpha} < 2$ such that $B_{j_{\alpha}}^\alpha$ is in $U$.
Finally, find $j<2$ such that $X:=\{\alpha<\mu^+\mid j_{\alpha}=j\}$ is in $U$.
Then $X$ is as sought.
\end{proof}

Fix a strictly increasing sequence $\langle \mu_i\mid i<\cf(\mu)\rangle$
of cardinals converging to $\mu$, with $\mu_0\ge\cf(\mu)$. For each $i<\cf(\mu)$, let $X_i$ and $j_i$ be given by the above claim.
By thinning out, we may also assume the existence of $j<2$ such that $j_i=j$ for all $i<\cf(\mu)$.

For every $\delta<\mu^+$, fix a decomposition $\delta=\biguplus_{i<\cf(\mu)}\Gamma_{\delta,i}$ such that $|\Gamma_{\delta,i}|\le\mu_i$ for all $i<\cf(\mu)$.
We shall now construct a matrix $\langle \beta_{\delta,i}\mid \delta<\mu^+, i<\cf(\mu)\rangle$
in such a way that, for each $\delta<\mu^+$, $\langle \beta_{\delta,i}\mid i<\cf(\mu)\rangle\in\prod_{i<\cf(\mu)}X_i$.
The definition is by recursion on $\delta<\mu^+$:

$\br$ For $\delta=0$, let $\beta_{\delta,i}:=\min(X_i)$ for all $i<\cf(\mu)$.

$\br$ For $\delta>0$ such that $\langle \beta_{\gamma,i}\mid \gamma<\delta, i<\cf(\mu)\rangle$ has already been defined,
since, for each $i<\cf(\mu)$, $x_{\delta,i}:=\{ \beta_{\gamma,i}\mid \gamma\in\Gamma_{\delta,i}\}$ is a subset of $X_i$
of size no more than $\mu_i$, we may pick $\beta_{\delta,i}\in X_i$ above $\sup\{ \beta_{\gamma,\iota}\mid \gamma<\delta, \iota<\cf(\mu)\}$
such that $c[x_{\delta,i}\times\{\beta_{\delta,i}\}]=\{j\}$.

This completes the construction.

For each $\delta<\mu^+$, let $a_\delta:=\{ \beta_{\delta,i}\mid i<\cf(\mu)\}$.
Evidently, $\langle a_\delta\mid \delta<\mu^+\rangle$ is $<$-increasing.
So, by the choice of the coloring $c$, we may now pick $(\gamma,\delta)\in[\mu^+]^2$ such that $c[a_\gamma\times a_\delta]=\{1-j\}$.
Find $i<\cf(\mu)$ such that $\gamma\in\Gamma_{\delta,i}$.
Then $\beta_{\gamma,i}\in x_{\delta,i}\cap a_\gamma$, $\beta_{\delta,i}\in a_\delta$,
and $c(\beta_{\gamma,i},\beta_{\delta,i})=j$. This is a contradiction.
\end{proof}

By \cite[Theorem~3.1]{paper15}, 
for every infinite regular cardinal $\mu$, if $\pl_6(\mu,\mu)$ holds 
then so does $\pr_1(\mu^+,\mu^+,\mu^+,\mu)$,
and 
for every infinite singular cardinal $\mu$, if $\pl_6(\mu^+,\mu)$ holds,
then so does $\pr_1(\mu^{++},\mu^{++},\mu^{++},\mu)$,
The same conclusions may be drawn from 
$\pr_6(\mu,\mu,\mu,\mu)$ and $\pr_6(\mu^+,\mu^+,\mu^+,\mu)$, respectively.
However, the upcoming series of results show that none of these instances are consistent.

\begin{prop} Let $\mu$ be a singular cardinal. 

Then $\pl_6(\mu^+,\mu)$ and $\pr_6(\mu^+,\mu^+,2,\mu)$ both fail.
\end{prop}
\begin{proof}  Towards a contradiction, suppose that $d:{}^{<\omega}\mu^+\rightarrow\omega$ witnesses $\pl_6(\mu^+,\mu)$
(resp.~$d:{}^{<\omega}\mu^+\rightarrow2$ witnesses $\pr_6(\mu^+,\mu^+,2,\mu)$).

\begin{claim} Let $A\in[\mu^+]^{<\mu}$ and $\beta<\mu^+$.
There exists $\sigma\in{}^{<\omega}\mu^+$ such that $d(\langle \alpha,\beta\rangle\conc\sigma)\neq 1$ for any $\alpha\in A$.
\end{claim}
\begin{proof} For all $\gamma<\mu^+$, let $u_\gamma:=\{ \langle\alpha,\beta,\gamma\rangle\mid \alpha\in A\rangle$
and $v_\gamma:=\{\langle\gamma\rangle\}$. 

$\br$ Assuming that $d:{}^{<\omega}\mu^+\rightarrow\omega$ witnesses $\pl_6(\mu^+,\mu)$,
we may now fix $(\gamma,\delta)\in[\mu^+]^2$ such that 
$d(\eta\conc\rho)=\ell(\eta)$ for all $\eta\in u_\gamma$ and $\rho\in v_\delta$.
Let $\sigma:=\langle \gamma,\delta\rangle$. Then, for every $\alpha\in\Gamma$,
$\langle \alpha,\beta,\gamma\rangle\in u_\gamma$ and $\langle\delta\rangle\in v_\delta$, so that 
$d(\langle \alpha,\beta\rangle\conc\sigma)=d(\langle\alpha,\beta,\gamma\rangle\conc\langle\delta\rangle)=3$.

$\br$ Assuming that $d:{}^{<\omega}\mu^+\rightarrow2$ witnesses $\pr_6(\mu^+,\mu^+,2,\mu)$,
we may now fix $(\gamma,\delta)\in[\mu^+]^2$ such that 
$d(\eta\conc\rho)=0$ for all $\eta\in u_\gamma$ and $\rho\in v_\delta$.
Let $\sigma:=\langle \gamma,\delta\rangle$. Then, for every $\alpha\in\Gamma$,
$\langle \alpha,\beta,\gamma\rangle\in u_\gamma$ and $\langle\delta\rangle\in v_\delta$, so that 
$d(\langle \alpha,\beta\rangle\conc\sigma)=0$.
\end{proof}

Let $\beta<\mu^+$. Fix a decomposition $\beta=\biguplus_{i<\cf(\mu)}A_{\beta,i}$ such that $|A_{\beta,i}|<\mu$ for all $i<\cf(\mu)$,
and then, for every $i<\cf(\mu)$, fix $\sigma_{\beta,i}\in{}^{<\omega}\mu^+$ such that $d(\langle \alpha,\beta\rangle\conc\sigma_{\beta,i})\neq1$ for any $\alpha\in A_{\beta,i}$.
Let $u_\beta:=\{\beta\}$ and $v_\beta:=\{ \langle\beta\rangle\conc\sigma_{\beta,i}\mid i<\cf(\mu)\}$ .

$\br$ Assuming that $d:{}^{<\omega}\mu^+\rightarrow\omega$ witnesses $\pl_6(\mu^+,\mu)$,
we may now fix $(\alpha,\beta)\in[\mu^+]^2$ such that 
$d(\eta\conc\rho)=\ell(\eta)$ for all $\eta\in u_\alpha$ and $\rho\in v_\beta$.
Find $i<\cf(\mu)$ such that $\alpha\in A_{\beta,i}$.
As $\langle\alpha\rangle\in u_\alpha$ and $\langle\beta\rangle\conc\sigma_{\beta,i}\in v_\beta$,
this must mean that $d(\langle\alpha\rangle\conc\langle\beta\rangle\conc\sigma_{\beta,i})=1$,
contradicting the fact that $\alpha\in A_{\beta,i}$.

$\br$ Assuming that $d:{}^{<\omega}\mu^+\rightarrow2$ witnesses $\pr_6(\mu^+,\mu^+,2,\mu)$,
we may now fix $(\alpha,\beta)\in[\mu^+]^2$ such that 
$d(\eta\conc\rho)=1$ for all $\eta\in u_\alpha$ and $\rho\in v_\beta$.
Find $i<\cf(\mu)$ such that $\alpha\in A_{\beta,i}$.
As $\langle\alpha\rangle\in u_\alpha$ and $\langle\beta\rangle\conc\sigma_{\beta,i}\in v_\beta$,
this must mean that $d(\langle\alpha\rangle\conc\langle\beta\rangle\conc\sigma_{\beta,i})=1$,
contradicting the fact that $\alpha\in A_{\beta,i}$.
\end{proof}

\begin{prop}\label{prop84} Let $\mu$ be an infinite cardinal. 
Then  $\pl_6(\mu,\mu)$ fails. Furthermore:
\begin{itemize}
\item[(a)] For every map $d:{}^{<\omega}\mu\rightarrow\omega$,
there exist a cardinal $\nu<\mu$ and two sequences $\langle u_\alpha\mid \alpha<\mu\rangle$ and $\langle v_\beta\mid \beta<\mu\rangle$,
with
\begin{enumerate}
\item $u_\alpha\s {}^{<\omega}\mu$, $|u_\alpha|=\nu$, and, for all $\varrho\in u_\alpha$, $\alpha\in\im(\varrho)$;
\item $v_\beta\s {}^{<\omega}\mu$, $|v_\beta|=1$, and, for all $\sigma\in v_\beta$,  $\langle \beta\rangle\sqsubseteq \sigma$,
\end{enumerate}
such that, for every $(\alpha,\beta)\in[\mu]^2$,
there are $\varrho\in u_\alpha$ and $\sigma\in v_\beta$
with $d(\varrho{}^\smallfrown\sigma)\neq\ell(\varrho)$.
\item[(b)] For every map $d:{}^{<\omega}\mu\rightarrow\omega$,
there exist two sequences $\langle u_\alpha\mid \alpha<\mu\rangle$ and $\langle v_\beta\mid \beta<\mu\rangle$,
with
\begin{enumerate}
\item $u_\alpha\s {}^{<\omega}\mu$, $|u_\alpha|=1$, and, for all $\varrho\in u_\alpha$, $\alpha\in\im(\varrho)$;
\item $v_\beta\s {}^{<\omega}\mu$, $|v_\beta|=|\beta|$, and, for all $\sigma\in v_\beta$,  $\langle \beta\rangle\sqsubseteq \sigma$,
\end{enumerate}
such that, for every $(\alpha,\beta)\in[\mu]^2$,
there are $\varrho\in u_\alpha$ and $\sigma\in v_\beta$
with $d(\varrho{}^\smallfrown\sigma)\neq\ell(\varrho)$.
\end{itemize}
\end{prop}
\begin{proof} (a) Suppose not, and let $d$ be a counterexample.

\begin{claim} Let $\beta<\mu$. There exists $\eta\in{}^{<\omega}\mu$ such that,
for all $\alpha<\beta$, $d(\langle\alpha,\beta\rangle^\smallfrown\eta)\neq1$.
\end{claim}
\begin{proof} For every $\gamma<\mu$, let $u_\gamma:=\{ \langle \alpha,\beta,\gamma\rangle\mid \alpha<\beta\}$.
For every $\delta<\mu$, let $v_\delta:=\{\langle\delta\rangle\}$. 
Now, by the choice of $d$ (using $\nu:=|\beta|$), there must exist $(\gamma,\delta)\in[\mu]^2$,
such that $d(\varrho{}^\smallfrown\sigma)=\ell(\varrho)$ for all $\varrho\in u_\gamma$ and $\sigma\in v_\delta$.
So, for all $\alpha<\beta$, $d(\langle\alpha,\beta,\gamma,\delta\rangle)=3$.
Set $\eta:=\langle\gamma,\delta\rangle$.
Then, for all $\alpha<\beta$, $d(\langle\alpha,\beta\rangle^\smallfrown\eta)=3$.
\end{proof}

For every $\alpha<\mu$, let $u_\alpha:=\{\langle\alpha\rangle\}$. 
For every $\beta<\mu$, pick $\eta_{\beta}\in{}^{<\omega}\mu$ as in the claim,
and let $v_\beta:=\{ \langle\beta\rangle^\smallfrown\eta_{\beta}\}$.
Now, by the choice of $d$ (using $\nu:=1$), there must exist $(\alpha,\beta)\in[\mu]^2$,
such that $d(\varrho{}^\smallfrown\sigma)=\ell(\varrho)$ for all $\varrho\in u_\alpha$ and $\sigma\in v_\beta$.
In particular, $d(\langle\alpha\rangle{}^\smallfrown\langle\beta\rangle{}^\smallfrown\eta_{\beta})=1$,
contradicting the choice of $\eta_{\beta}$.

(b) Left to the reader (but see the proof of Proposition~\ref{prop33}(b)).
\end{proof}

\begin{prop}\label{prop33} Let $\mu$ be an infinite cardinal. 
Then $\pr_6(\mu,\mu,2,\mu)$ fails. Furthermore:
\begin{itemize}
\item[(a)] For every map $d:{}^{<\omega}\mu\rightarrow2$,
there exist a cardinal $\nu<\mu$, $i<2$, and two sequences $\langle u_\alpha\mid \alpha<\mu\rangle$ and $\langle v_\beta\mid \beta<\mu\rangle$,
with
\begin{enumerate}
\item $u_\alpha\s {}^{<\omega}\mu$, $|u_\alpha|=\nu$, and, for all $\varrho\in u_\alpha$, $\alpha\in\im(\varrho)$;
\item $v_\beta\s {}^{<\omega}\mu$, $|v_\beta|=1$, and, for all $\sigma\in v_\beta$,  $\beta\in\im(\sigma)$,
\end{enumerate}
such that, for every $(\alpha,\beta)\in[\mu]^2$,
there are $\varrho\in u_\alpha$ and $\sigma\in v_\beta$
with $d(\varrho{}^\smallfrown\sigma)\neq i$.
\item[(b)] For every map $d:{}^{<\omega}\mu\rightarrow2$,
there exist $i<2$ and two sequences $\langle u_\alpha\mid \alpha<\mu\rangle$ and $\langle v_\beta\mid \beta<\mu\rangle$,
with
\begin{enumerate}
\item $u_\alpha\s {}^{<\omega}\mu$, $|u_\alpha|=1$, and, for all $\varrho\in u_\alpha$, $\alpha\in\im(\varrho)$;
\item $v_\beta\s {}^{<\omega}\mu$, $|v_\beta|=|\beta|$, and, for all $\sigma\in v_\beta$, $\beta\in\im(\sigma)$,
\end{enumerate}
such that, for every $(\alpha,\beta)\in[\mu]^2$,
there are $\varrho\in u_\alpha$ and $\sigma\in v_\beta$
with $d(\varrho{}^\smallfrown\sigma)\neq i$.
\end{itemize}
\end{prop}
\begin{proof} (a) Left to the reader (but see the proof of Proposition~\ref{prop84}(a)).

(b)  Suppose not, and let $d$ be a counterexample.

\begin{claim} Let $(\alpha,\beta)\in[\mu]^2$. There exists $\eta\in{}^{<\omega}\mu$ such that $d(\langle\alpha,\beta\rangle^\smallfrown\eta)\neq1$.
\end{claim}
\begin{proof} For every $\gamma<\mu$, let $u_\gamma:=\{ \langle \alpha,\beta,\gamma\rangle\}$.
For every nonzero $\delta<\mu$, let $v_\delta:=\{\langle\delta\rangle\}$. 
By the choice of $d$, there must exist $(\gamma,\delta)\in[\mu]^2$,
such that $d(\varrho{}^\smallfrown\sigma)=0$ for all $\varrho\in u_\gamma$ and $\sigma\in v_\delta$.
Set $\eta:=\langle\gamma,\delta\rangle$.
Then, $d(\langle\alpha,\beta\rangle^\smallfrown\eta)=0$.
\end{proof}

For each $(\alpha,\beta)\in[\mu]^2$, let $\eta_{\alpha,\beta}\in{}^{<\omega}\mu$ be given by the claim.
For every $\alpha<\mu$, let $u_\alpha:=\{\langle\alpha\rangle\}$. 
For every $\beta<\mu$, let $v_\beta:=\{ \langle\beta\rangle^\smallfrown\eta_{\alpha,\beta}\mid \alpha<\beta\}$.
By the choice of $d$, pick $(\alpha,\beta)\in[\mu]^2$ such that $d(\varrho{}^\smallfrown\sigma)=1$ for all $\varrho\in u_\alpha$ and $\sigma\in v_\beta$.
Then $d(\langle\alpha\rangle{}^\smallfrown\langle\beta\rangle{}^\smallfrown\eta_{\alpha,\beta})=1$,
contradicting the choice of $\eta_{\alpha,\beta}$.
\end{proof}

\section{$C$-sequences}\label{cseqsection}
\begin{defn} A \emph{$C$-sequence over $\kappa$}
is  sequence $\vec C=\langle C_\alpha\mid\alpha<\kappa\rangle$
such that, for all $\alpha<\kappa$, $C_\alpha$ is a closed subset of $\alpha$ with $\sup(C_\alpha)=\sup(\alpha)$.
\end{defn}

\begin{defn} A $C$-sequence $\langle C_\alpha\mid\alpha<\kappa\rangle$ is said to \emph{avoid} a set $\Gamma$
iff $\acc(C_\alpha)\cap\Gamma=\emptyset$ for all $\alpha<\kappa$.
\end{defn}

Note that a stationary subset $\Gamma$ of $\kappa$ is nonreflecting iff there exists a $C$-sequence over $\kappa$ that avoids it.

In this paper, we shall make use of two instances of the parameterized proxy principle from \cite{paper22,paper23}.
The first instance reads as follows (see \cite[Definition~1.3]{paper22}):

\begin{definition}\label{xbox}
$\boxtimes^-(\kappa)$ asserts
the existence of a $C$-sequence $\langle C_\alpha\mid\alpha<\kappa\rangle$ such that:
\begin{itemize}
\item for all $\alpha < \kappa$ and $\delta\in\acc(C_\alpha)$, $C_\alpha\cap\delta=C_{\delta}$;
\item for every cofinal $B\s\kappa$, there exist stationarily many $\alpha<\kappa$ such that $\sup(\nacc(C_\alpha)\cap B)=\alpha$.
\end{itemize}
\end{definition}

The second instance reads as follows (see \cite[Definition~4.10, Theorem~4.15(iii) and Convention 4.18]{paper23}):
\begin{definition}\label{proxydef} For a stationary subset $S\s\kappa$,
 $\p^-(\kappa,\kappa^+,\allowbreak{\sq^*},1,\{S\},2)$ asserts the existence of a $C$-sequence $\langle C_\alpha\mid\alpha<\kappa\rangle$ 
and a stationary subset $\Delta\s S$ such that:
\begin{itemize}
\item   for all $\alpha<\kappa$ and $\delta\in\acc(C_\alpha)\cap\Delta$, $\sup((C_\alpha\cap\delta)\symdiff C_{\delta})<\delta$;
\item  for every cofinal  $B\s\kappa$,
there exist stationarily many $\alpha \in \Delta$ such that:
$$\sup\{\varepsilon\in B\cap\alpha\mid \min(C_\alpha\setminus(\varepsilon+1))\in B\}=\alpha.$$
\end{itemize}
\end{definition}

\subsection{Walks on ordinals}
For the rest of this subsection, 
let us fix a $C$-sequence $\vec C=\langle C_\alpha\mid\alpha<\kappa\rangle$ over $\kappa$.
The next definition is due to Todorcevic; see \cite{TodWalks} for a comprehensive treatment.
\begin{defn}[Todorcevic] From $\vec C$, derive maps $\Tr:[\kappa]^2\rightarrow{}^\omega\kappa$,
$\rho_2:[\kappa]^2\rightarrow	\omega$,
$\tr:[\kappa]^2\rightarrow{}^{<\omega}\kappa$ and $\lambda:[\kappa]^2\rightarrow\kappa$, as follows.
Let $(\alpha,\beta)\in[\kappa]^2$ be arbitrary.
\begin{itemize}
\item $\Tr(\alpha,\beta):\omega\rightarrow\kappa$ is defined by recursion on $n<\omega$:
$$\Tr(\alpha,\beta)(n):=\begin{cases}
\beta,&n=0\\
\min(C_{\Tr(\alpha,\beta)(n-1)}\setminus\alpha),&n>0\ \&\ \Tr(\alpha,\beta)(n-1)>\alpha\\
\alpha,&\text{otherwise}
\end{cases}
$$
\item $\rho_2(\alpha,\beta):=\min\{n<\omega\mid \Tr(\alpha,\beta)(n)=\alpha\}$;
\item $\tr(\alpha,\beta):=\Tr(\alpha,\beta)\restriction \rho_2(\alpha,\beta)$;
\item $\lambda(\alpha,\beta):=\max\{ \sup(C_{\Tr(\alpha,\beta)(i)}\cap\alpha) \mid i<\rho_2(\alpha,\beta)\}$.
\end{itemize}
\end{defn}

The next three facts are quite elementary. See \cite{paper44} for details.
\begin{fact}\label{fact1} Whenever $0<\delta<\beta<\kappa$, if $\delta\notin\bigcup_{\alpha<\kappa}\acc(C_\alpha)$, then $\lambda(\delta,\beta)<\delta$.
\end{fact}
\begin{fact}\label{fact2} Whenever $\lambda(\delta,\beta)<\alpha<\delta<\beta<\kappa$, 
$\tr(\alpha,\beta)=\tr(\delta,\beta){}^\smallfrown \tr(\alpha,\delta)$.
\end{fact}

\begin{fact}\label{fact3}  Whenever $\alpha<\delta<\beta<\kappa$ with  $\delta\in\im(\tr(\alpha,\beta))$,
$$\lambda(\alpha,\beta)=\max\{\lambda(\delta,\beta),\lambda(\alpha,\delta)\}.$$
\end{fact}

\begin{defn}[\cite{paper44}] For every $(\alpha,\beta)\in[\kappa]^2$, we define an ordinal $\last{\alpha}{\beta}\in[\alpha,\beta]$ via:
$$\last{\alpha}{\beta}:=\begin{cases}
\alpha,&\text{if }\lambda(\alpha,\beta)<\alpha;\\
\min(\im(\tr(\alpha,\beta)),&\text{otherwise}.
\end{cases}$$
\end{defn}

\begin{fact}[\cite{paper44}]\label{last} Let $(\alpha,\beta)\in[\kappa]^2$ with $\alpha>0$. Then
\begin{enumerate}
\item $\lambda(\last{\alpha}{\beta},\beta)<\alpha$;\footnote{Recall Convention~\ref{conv28}.}
\item If $\last{\alpha}{\beta}\neq\alpha$, then $\alpha\in\acc(C_{\last{\alpha}{\beta}})$;
\item $\tr(\last{\alpha}{\beta},\beta)\sq\tr(\alpha,\beta)$.
\end{enumerate}
\end{fact}

For the purpose of this paper, we also introduce the following ad-hoc notation.
\begin{defn}\label{etanotation}
For every ordinal $\eta<\kappa$ and a pair $(\alpha,\beta)\in[\kappa]^2$, we let
$$\eta_{\alpha,\beta}:=\min\{ n<\omega\mid \eta\in C_{\Tr(\alpha,\beta)(n)}\text{ or }n=\rho_2(\alpha,\beta)\}+1.$$
\end{defn}

We conclude this subsection by proving a useful lemma.
\begin{lemma}\label{Lemma212} For ordinals $\eta<\alpha<\delta<\beta<\kappa$, 
if $\lambda(\delta,\beta)=\eta$ and $\rho_2(\delta,\beta)=\eta_{\delta,\beta}$,
then $\tr(\alpha,\beta)(\eta_{\alpha,\beta})=\delta$.
\end{lemma}
\begin{proof} Under the above hypothesis, Fact~\ref{fact2} entails that $\tr(\alpha,\beta)=\tr(\delta,\beta){}^\smallfrown \tr(\alpha,\delta)$.
As $\eta_{\delta,\beta}=\rho_2(\delta,\beta)<\rho_2(\delta,\beta)+1$, it altogether follows that 
$$\begin{array}{lll}\eta_{\alpha,\beta}&=&\min\{ n<\omega\mid \eta\in C_{\Tr(\alpha,\beta)(n)}\text{ or }n=\rho_2(\alpha,\beta)\}+1\\
&=&\min\{ n<\omega\mid \eta\in C_{\Tr(\delta,\beta)(n)}\}+1\\
&=&\rho_2(\delta,\beta),\end{array}$$
so that $\tr(\alpha,\beta)(\eta_{\alpha,\beta})=\delta$.
\end{proof}

\subsection{Cardinal characteristics of $C$-sequences}

\begin{defn} For a $C$-sequence $\vec C=\langle C_\alpha\mid\alpha<\kappa\rangle$ and a subset $\Gamma\s\kappa$:
\begin{itemize}
\item $\chi_1(\vec C)$ is the supremum of $\sigma+1$ over all $\sigma<\kappa$ satisfying the following.
For every pairwise disjoint subfamily $\mathcal B\s[\kappa]^{\sigma}$ of size $\kappa$, there are a stationary set $\Delta\s\kappa$ and an ordinal $\eta<\kappa$
such that, for every $\delta\in\Delta$, there exist $\kappa$ many $b\in\mathcal B$ such that,
for every $\beta\in b$,
$\lambda(\delta,\beta)=\eta$ and $\rho_2(\delta,\beta)=\eta_{\delta,\beta}$.
\item $\chi_2(\vec C,\Gamma)$ is the supremum of $\sigma+1$ over all $\sigma<\kappa$ satisfying the following.
For every pairwise disjoint subfamily $\mathcal B\s[\kappa]^{\sigma}$ of size $\kappa$, there are club many $\delta\in\Gamma$ 
for which there exist an ordinal $\eta<\delta$ and $\kappa$ many $b\in\mathcal B$ such that,
for every $\beta\in b$,
$\lambda(\delta,\beta)=\eta$ and $\rho_2(\delta,\beta)=\eta_{\delta,\beta}$.
\end{itemize}
\end{defn}

Note that if $\Gamma$ is stationary, then $\chi_2(\vec C,\Gamma)\le\chi_1(\vec C)$.

\begin{lemma}\label{lemma45}  Suppose that $\Gamma\s\kappa$ is a nonreflecting stationary set, and $\kappa\ge\aleph_2$.
Then there exists a $C$-sequence $\vec C$ that avoids $\Gamma$
such that $$\chi_1(\vec C)\ge\sup\{\sigma<\kappa\mid \Gamma\cap E^\kappa_{>\sigma}\text{ is stationary}\}.$$
\end{lemma}
\begin{proof} As $\Gamma$ is nonreflecting, we commence by fixing a $C$-sequence $\vec C=\langle C_\alpha\mid\alpha<\kappa\rangle$ that avoid $\Gamma$.
It is clear that $\vec C$ is \emph{amenable} in the sense of \cite[Definition~1.3]{paper29}.
Let $\Lambda:=\{\sigma<\kappa\mid \Gamma\cap E^\kappa_{>\sigma}\text{ is stationary}\}$.
It is clear that $\Lambda$ is some limit ordinal $\le\kappa$, and that if $\Lambda=0$, there is nothing that needs to be done, so we assume $\Lambda>0$.
For each $\sigma\in\Lambda$, $\Omega^\sigma:=\Gamma\cap E^\kappa_{>\sigma}$ is stationary.
By \cite[Lemma~1.15]{paper29}, we may now fix a conservative postprocessing function $\Phi:\mathcal K(\kappa)\rightarrow\mathcal K(\kappa)$,
a cofinal subset $\Sigma\s\Lambda$ and an injection $h:\Sigma\rightarrow\kappa$ such that $\{\alpha\in\Omega^\sigma\mid \min(\Phi(C_\alpha))=h(\sigma)\}$ is stationary for all $\sigma\in\Sigma$.
In simple words, this means that there exists a $C$-sequence $\vec D=\langle D_\alpha\mid \alpha<\kappa\rangle$ such that:
\begin{itemize}
\item $D_\alpha\s C_\alpha$ for all $\alpha<\kappa$;
\item $\Gamma^{\sigma}:=\{\alpha\in \Gamma\cap E^\kappa_{>\sigma}\mid \min(D_\alpha)=h(\sigma)\}$ is stationary for all $\sigma\in\Sigma$.
\end{itemize}

Note that since $h$ is injective, $\langle \Gamma^{\sigma}\mid \sigma\in\Sigma\rangle$ consists of pairwise disjoint stationary sets.
For each $\sigma\in\Sigma$, since $\vec D\restriction\Gamma^\sigma$ is an amenable $C$-sequence, 
it follows from \cite[Lemma~2.2 and Fact~2.4(2)]{paper29} that there exists a $C$-sequence 
$\langle C_\alpha^\bullet\mid \alpha\in\Gamma^\sigma\rangle$ such that:
\begin{itemize}
\item $\acc(C_\alpha^\bullet)\s\acc(D_\alpha)$ for all $\alpha\in\Gamma^\sigma$;
\item For every club $D\s\kappa$, there exists $\alpha\in\Gamma^\sigma$ with $\sup(\nacc(C^\bullet_\alpha)\cap D)=\alpha$.
\end{itemize}

For every $\alpha\in\kappa\setminus\bigcup_{\sigma\in\Sigma}\Gamma^\sigma$, let $C_\alpha^\bullet:=D_\alpha$.
Recalling that $\Sigma$ is a cofinal subset of $\Lambda$, we altogether infer that:
\begin{itemize}
\item $\acc(C_\alpha^\bullet)\cap\Gamma=\emptyset$ for all $\alpha<\kappa$;
\item For every club $D\s\kappa$ and every $\sigma<\kappa$ such that $\Gamma\cap E^\kappa_{>\sigma}$ is stationary, there exists $\alpha\in \Gamma\cap E^\kappa_{>\sigma}$ with $\sup(\nacc(C^\bullet_\alpha)\cap D)=\alpha$.
\end{itemize}

We now walk along $\vec C^\bullet:=\langle C_\alpha^\bullet\mid\alpha<\kappa\rangle$, and verify that it is as sought.
For this, suppose that $\sigma<\kappa$ is such that $\Gamma\cap E^\kappa_{>\sigma}$ is stationary,
and that we are given a pairwise disjoint subfamily $\mathcal B\s[\kappa]^{\sigma}$ of size $\kappa$.
It suffices to prove that for every club $D\s\kappa$,
there exists $\delta\in D$, such that, for $\kappa$ many $b\in\mathcal B$, for some $\eta<\delta$, 
for every $\beta\in b$,
$\lambda(\delta,\beta)=\eta$ and $\rho_2(\delta,\beta)=\eta_{\delta,\beta}$.

Thus, let $D$ be an arbitrary club in $\kappa$. Without loss of generality, $D\s\acc(\kappa)$.
Pick $\gamma\in \Gamma\cap E^\kappa_{>\sigma}$ with $\sup(\nacc(C^\bullet_\gamma)\cap D)=\gamma$.
\begin{claim} Let $b\in\mathcal B$ with $\min(b)>\gamma$.
There are $\delta\in D\cap\gamma$ and $\eta<\delta$ such that,
for every $\beta\in b$,
$\lambda(\delta,\beta)=\eta$ and $\rho_2(\delta,\beta)=\eta_{\delta,\beta}$.
\end{claim}
\begin{proof} Set $\Lambda:=\sup\{\lambda(\gamma,\beta)\mid \beta\in b\}$.
As $|b|<\cf(\gamma)$ and $\gamma\in\Gamma$, it follows from Fact~\ref{fact1} that $\Lambda<\gamma$.
Now pick $\delta\in\nacc(C_\gamma^\bullet)\cap D$ for which $\eta:=\sup(C_\gamma^\bullet\cap\delta)$ is bigger than $\Lambda$.
For every $\beta\in b$, 
$\lambda(\gamma,\beta)\le\Lambda<\delta<\gamma<\beta$, 
so by Facts \ref{fact2} and \ref{fact3},
$\lambda(\delta,\beta)=\max\{\lambda(\gamma,\beta),\lambda(\delta,\gamma)\}$.
As $\lambda(\gamma,\beta)\le\Lambda<\eta=\lambda(\delta,\gamma)$, it follows that $\lambda(\delta,\beta)=\eta$ and $\rho_2(\delta,\beta)=\eta_{\delta,\beta}$.
\end{proof}

As $|D\cap\gamma|<\kappa=|\mathcal B|$, there must be $\delta\in D\cap\gamma$ and $\eta<\delta$ such that, 
for $\kappa$ many $b\in\mathcal B$, for some $\eta<\delta$, 
for every $\beta\in b$,
$\lambda(\delta,\beta)=\eta$ and $\rho_2(\delta,\beta)=\eta_{\delta,\beta}$.
\end{proof}

It follows that for every regular uncountable cardinal $\mu$, there exists a $C$-sequence $\vec C$ over $\mu^+$ with $\chi_1(\vec C)\ge\mu$.
As made clear by the proof of \cite[Lemma~2.4]{paper13},
for every singular cardinal $\mu$ of uncountable cofinality, there exists a $C$-sequence $\vec C$ over $\mu^+$ with $\chi_1(\vec C)\ge\cf(\mu)$.
This raises the following question:
\begin{q} Suppose that $\mu$ is an infinite cardinal of countable cofinality. Must there exist a $C$-sequence $\vec C$ over $\mu^+$ with $\chi_1(\vec C)\ge\omega$?
\end{q}

\begin{lemma}\label{Lemma32} 
\begin{enumerate}
\item If $\square(\kappa)$ holds and $\kappa\ge\aleph_2$, then there exists a $C$-sequence $\vec C$ over $\kappa$
such that $\chi_1(\vec C)=\sup(\reg(\kappa))$;
\item If $\boxtimes^-(\kappa)$ holds,  then there exists a $C$-sequence $\vec C$ over $\kappa$
such that $\chi_2(\vec C,\kappa)=\sup(\reg(\kappa))$.
\end{enumerate}
\end{lemma}
\begin{proof} In Case~(1), since $\square(\kappa)$ holds and $\kappa\ge\aleph_2$, by \cite[Proposition~3.5]{paper24}, we may fix a $C$-sequence $\vec C=\langle C_\alpha\mid\alpha<\kappa\rangle$ satisfying the two:
\begin{enumerate}
\item[($\aleph$)] for every $\alpha<\kappa$ and $\delta\in\acc(\kappa)$, $C_{\delta}=C_\alpha\cap\delta$;
\item[($\beth$)] for every club $D\s\kappa$, there exists $\gamma>0$ with $\sup(\nacc(C_\gamma)\cap D)=\gamma$.
\end{enumerate}

In Case~(2), just let $\vec C=\langle C_\alpha\mid\alpha<\kappa\rangle$ be a $\boxtimes^-(\kappa)$-sequence.

Let $X_1$ denote an arbitrary club in $\kappa$,
and let $X_2$ denote an arbitrary stationary set in $\kappa$.
Let $n\in\{1,2\}$.
To verify Clause~$(n)$, 
we shall prove that given $\sigma<\sup(\reg(\kappa))$
and a pairwise disjoint subfamily $\mathcal B\s[\kappa]^{\sigma}$ of size $\kappa$,
there exists $\delta\in X_n$, such that, for $\kappa$ many $b\in\mathcal B$, for some $\eta<\delta$, 
for every $\beta\in b$,
$\lambda(\delta,\beta)=\eta$ and $\rho_2(\delta,\beta)=\eta_{\delta,\beta}$.

Without loss of generality, $X_n\s\acc(\kappa)$.
For every $\tau<\kappa$, fix $b_\tau\in\mathcal B$ with $\min(b_\tau)>\tau$.
Define a function $f:E^\kappa_{>\sigma}\rightarrow\kappa$ via
$$f(\tau):=\sup\{\lambda(\last{\tau}{\beta},\beta)\mid \beta\in b_\tau\}.$$

As $|b_\tau|<\cf(\tau)$, Fact~\ref{last}(1) entails that $f$ is regressive.
So, fix a stationary $T\s E^\kappa_{>\sigma}$ such that $f\restriction T$ is constant with value, say, $\zeta$.
Now, if $n=1$, then  using Clause~($\beth$), we may pick a nonzero ordinal $\gamma<\kappa$ with $\sup(\nacc(C_\gamma)\cap (X_1\setminus\zeta))=\gamma$,
and if $n=2$, then we may pick a nonzero ordinal $\gamma<\kappa$ with $\sup(\nacc(C_\gamma)\cap (X_2\setminus\zeta))=\gamma$.
\begin{claim} Let $\tau\in T$ above $\gamma$. 
There are $\delta\in X_n\cap\gamma$ and $\eta<\delta$ such that,
for every $\beta\in b_\tau$,
$\lambda(\delta,\beta)=\eta$ and $\rho_2(\delta,\beta)=\eta_{\delta,\beta}$.
\end{claim}
\begin{proof} By Fact~\ref{last}(1), the following ordinal is smaller than $\gamma$:
$$\zeta':=\begin{cases}
0,&\text{if }\gamma\in\acc(C_{\tau});\\
\sup(C_{\tau}\cap\gamma),&\text{if }\gamma\in\nacc(C_{\tau});\\
\lambda(\last{\gamma}{\tau},\tau),&\text{otherwise}.
\end{cases}$$
Thus, we may pick a large enough $\delta\in\nacc(C_\gamma)\cap X_n$ such that $\sup(C_\gamma\cap\delta)>\max\{\zeta,\zeta'\}$.
Denote $\eta:=\sup(C_\gamma\cap\delta)$, so that $\eta<\delta$.
Let $\beta\in b_\tau$ be arbitrary.
We have
$$\lambda(\last{\tau}{\beta},\beta)\le f(\tau)\le\max\{\zeta,\zeta'\}<\eta<\eta+1<\delta<\gamma<\tau<\beta.$$
We shall show that $\lambda(\delta,\beta)=\eta$ and $\rho_2(\delta,\beta)=\eta_{\delta,\beta}$.

By Fact~\ref{fact3}, $\lambda(\delta,\beta)=\max\{\lambda(\last{\tau}{\beta},\beta),\lambda(\delta,\last{\tau}{\beta})\}$.
Now, there are three cases to consider:

$\br$ If $\gamma\in\acc(C_{\tau})$, then $C_{\tau}\cap\gamma=C_\gamma$,
and since $\delta\in C_\gamma$,
$\tr(\delta,\beta)=\tr(\last{\tau}{\beta},\beta){}^\smallfrown\langle \last{\tau}{\beta}\rangle$,
and $\lambda(\delta,\last{\tau}{\beta})=\sup(C_\gamma\cap\delta)=\eta>\zeta\ge\lambda(\last{\tau}{\beta},\beta)$,
so the conclusion follows.

$\br$ If $\gamma\in\nacc(C_{\tau})$, then, since $\delta\in C_\gamma$,
$\tr(\delta,\beta)=\tr(\last{\tau}{\beta},\beta){}^\smallfrown\langle \last{\tau}{\beta},\gamma\rangle$,
so that $\lambda(\delta,\beta)=\max\{\lambda(\last{\tau}{\beta},\beta),\sup(C_{\last{\tau}{\beta}}\cap\delta),\sup(C_\gamma\cap\delta)\}=
\max\{\lambda(\last{\tau}{\beta},\beta),\zeta',\eta\}$,
and the conclusion follows.

$\br$ Otherwise, $\last{\gamma}{\tau}\neq\tau$.
Then $\lambda(\last{\gamma}{\tau},\tau)=\zeta'<\delta<\gamma\le\last{\gamma}{\tau}<\tau$,
and so, by Fact~\ref{fact2},
$\tr(\delta,\tau)=\tr(\last{\gamma}{\tau},\tau){}^\smallfrown\tr(\delta,\last{\gamma}{\tau})$.
Thus,  by Fact~\ref{fact3},
$$\lambda(\delta,\tau)=\max\{\lambda(\last{\gamma}{\tau},\tau),\lambda(\delta,\last{\gamma}{\tau})\}=\max\{\zeta',\lambda(\delta,\last{\gamma}{\tau})\}.$$

By Clause~(1) together with Fact~\ref{last}(2),
$\lambda(\delta,\last{\tau}{\beta})=\lambda(\delta,\tau)$.
As $\delta\in C_\gamma\cap \last{\gamma}{\tau}=C_{\last{\gamma}{\tau}}$, we get that $\lambda(\delta,\last{\gamma}{\tau})=\sup(C_\gamma\cap\delta)=\eta$.
Altogether, $\lambda(\delta,\beta)=\max\{\lambda(\last{\tau}{\beta},\beta),\zeta',\eta\}$.
But, $\eta>\max\{\zeta,\zeta'\}\ge\{\lambda(\last{\tau}{\beta},\beta),\zeta'\}$,
and the conclusion follows.
\end{proof}
As $|X_n\cap\gamma|<\kappa=|T|$, there must be $\delta\in X_n\cap\gamma$ and $\eta<\delta$ such that, 
for $\kappa$ many $b\in\mathcal B$, for some $\eta<\delta$, 
for every $\beta\in b$,
$\lambda(\delta,\beta)=\eta$ and $\rho_2(\delta,\beta)=\eta_{\delta,\beta}$.
\end{proof}

\begin{lemma}\label{lemma316}
Suppose that $S\s\kappa$ is a stationary set and $\vec C$ is a witness to $\p^-(\kappa,\kappa^+,\allowbreak{\sq^*},1,\{S\},2)$.
For every cardinal $\chi$ such that $S\cap E^\kappa_{<\chi}$ is nonstationary,
$\chi_2(\vec C,\kappa)\ge\chi$.
\end{lemma}
\begin{proof} 
Write $\vec C$ as $\langle C_\alpha\mid\alpha<\kappa\rangle$.
Fix a stationary subset $\Delta\s S$ such that:
\begin{enumerate}
\item   for all $\alpha<\kappa$ and $\delta\in\acc(C_\alpha)\cap\Delta$, $\sup((C_\alpha\cap\delta)\symdiff C_{\delta})<\delta$;
\item  for every cofinal  $B\s\kappa$,
there exist stationarily many $\alpha \in \Delta$ such that:
$$\sup\{\varepsilon\in B\cap\alpha\mid \min(C_\alpha\setminus(\varepsilon+1))\in B\}=\alpha.$$
\end{enumerate}

Now, suppose that  $\chi$ is cardinal such that $S\cap E^\kappa_{<\chi}$ is nonstationary.
Let $\sigma<\chi$, let $\mathcal B\s[\kappa]^{\sigma}$ be a pairwise disjoint family of size $\kappa$,
and let $\Gamma$ be an arbitrary stationary subset of $\kappa$;
we shall show that there exist $\gamma\in\Gamma$ and $\mathcal B'\in[\mathcal B]^\kappa$
such that, for every $b\in\mathcal B'$, there exists $\eta<\gamma$,
such that, for every $\beta\in b$,
$\lambda(\gamma,\beta)=\eta$ and $\rho_2(\gamma,\beta)=\eta_{\gamma,\beta}$.

Using Clause~(2),
fix $\delta\in\Delta\cap E^\kappa_{>\sigma}$ such that $\sup(\nacc(C_\delta)\cap\Gamma)=\delta$,
and then set $\mathcal B':=\{ b\in\mathcal B\mid \min(b)>\delta\}$.

Now, let $b\in\mathcal B'$ be arbitrary.
By Fact~\ref{last}(1) and as 
$\cf(\delta)>\sigma=\otp(b)$, $\Lambda:=\sup_{\beta\in b}\lambda(\last{\delta}{\beta},\beta)$ is $<\delta$.
Using Clause~(1) and Fact~\ref{last}(2), 
$$\Lambda':=\sup_{\beta\in b}\min\{\epsilon\in C_\delta\setminus\Lambda\mid C_\delta\cap[\epsilon,\delta)=C_{\last{\delta}{\beta}}\cap[\epsilon,\delta)\}$$ is $<\delta$, as well.
Fix a large enough $\gamma\in \nacc(C_\delta)\cap\Gamma$ for which $\eta:=\sup(C_\delta\cap\gamma)$ is $>\Lambda'$.
Note that $\gamma\in\Gamma\cap\delta$ and $\eta<\gamma$.
Let $\beta\in b$ be arbitrary. We have:
$$\lambda(\last{\delta}{\beta},\beta)\le\Lambda<\gamma<\delta\le\last{\delta}{\beta}\le\beta,$$
so, by Facts \ref{fact2} and \ref{fact3}, $\lambda(\gamma,\beta)=\max\{\lambda(\last{\delta}{\beta},\beta),\lambda(\gamma,\last{\delta}{\beta})\}$.
Fix $\epsilon\in C_\delta\cap[\Lambda,\gamma)$ such that
$C_\delta\cap[\epsilon,\delta)=C_{\last{\delta}{\beta}}\cap[\epsilon,\delta)$.
As $\gamma>\epsilon$, we have that $\gamma\in\nacc(C_{\last{\delta}{\beta}})$
and $\lambda(\gamma,\last{\delta}{\beta})=\sup(C_{\last{\delta}{\beta}}\cap\gamma)\ge\epsilon\ge\Lambda>\lambda(\last{\delta}{\beta},\beta)$.
Altogether, $\lambda(\gamma,\beta)=\sup(C_{\last{\delta}{\beta}}\cap\gamma)=\sup(C_\delta\cap\gamma)=\eta$.
\end{proof}

\begin{lemma}\label{done} For a cardinal $\chi$, assume any of the two:
\begin{enumerate}
\item $\pl_1(\kappa,1,\chi)$ holds, or
\item There exists a $C$-sequence $\vec C$ over $\kappa$ with $\chi_1(\vec C)\ge\chi$.
\end{enumerate}

Then there exists a coloring $d_1:[\kappa]^2\rightarrow\kappa$ \emph{good for $\chi$} in the following sense.
For all $\sigma<\chi$ and a pairwise disjoint subfamily $\mathcal B\s[\kappa]^{\sigma}$ of size $\kappa$, there are a stationary set $\Delta\s\kappa$ and an ordinal $\eta<\kappa$
such that, for every $\delta\in\Delta$, there exists $b\in\mathcal B$ with $\min(b)>\max\{\eta,\delta\}$ satisfying $d_1[\{\eta\}\times b]=\{\delta\}$.
\end{lemma}
\begin{proof} (1) Suppose that $\mathbf t:[\kappa]^2\rightarrow[\kappa]^3$ witnesses $\pl_1(\kappa,1,\chi)$.
Define $d_1:[\kappa]^2\rightarrow\kappa$ by letting $d_1(\alpha,\beta):=\beta^*$ whenever $\mathbf t(\alpha,\beta)=(\tau,\alpha^*,\beta^*)$.
It is clear that $d_1$ is as sought.

(2) Walk along such a $\vec C$.
Now, pick any coloring $d_1:[\kappa]^2\rightarrow\kappa$ such that, for every $\eta,\beta<\kappa$ with $\eta+1<\beta$,
$d_1(\eta,\beta)=\Tr(\eta+1,\beta)(\eta_{\eta+1,\beta})$.
By Lemma~\ref{Lemma212}, $d_1$ is as sought.
\end{proof}

We conclude this section with an improvement of \cite[Lemma~2.16]{paper44}.

\begin{defn}[\cite{paper35}]For a $C$-sequence $\vec{C} = \langle C_\alpha \mid \alpha < \kappa \rangle$,
      $\chi(\vec{C})$ is the least cardinal $\chi \leq \kappa$ such that there exists
      $\Delta \in [\kappa]^\kappa$ with the property that, for every $\epsilon<\kappa$,
      for some $a\in[\kappa]^\chi$,        $\Delta\cap\epsilon\s\bigcup_{\alpha\in a}C_\alpha$.
      
If $\kappa$ is weakly compact, then $\chi(\kappa)$ is defined to be $0$;\footnote{$\chi(\kappa)$ should be understood 
as a measure of how far $\kappa$ is from being weakly compact.
By \cite[Theorem~6.3.5]{TodWalks}, 
if $\kappa$ is weakly compact,
then $\chi(\vec C)=1$ for every $C$-sequence $\vec C$ over $\kappa$.}
 otherwise,
$\chi(\kappa)$ is the supremum of $\chi(\vec C)$ over all $C$-sequences $\vec C=\langle C_\alpha\mid\alpha<\kappa\rangle$.
\end{defn}
\begin{lemma} Suppose that $\kappa$ is an inaccessible cardinal. 

For any cardinal $\chi$ such that there exists a coloring $d_1$ good for $\chi$ in the sense of Lemma~\ref{done},
$\chi(\kappa)\ge\chi$.
In particular, if $\pl_1(\kappa,1,\chi)$ holds, then $\chi(\kappa)\ge\chi$.
\end{lemma}
\begin{proof} Suppose $d_1:[\kappa]^2\rightarrow\kappa$ is a coloring good for $\chi$.
By a straight-forward modification, we may assume that $d_1(\eta,\beta)<\beta$ for all $\eta<\beta<\kappa$.
Denote $\Sigma:=\{\alpha<\kappa\mid \cf(\alpha)<\alpha\}$.
Now, define a $C$-sequence $\vec C=\langle C_\alpha\mid\alpha<\kappa\rangle$, as follows.

$\br$ Set $C_0:=\emptyset$ and $C_\omega:=\omega$;

$\br$ For every $\alpha\in\Sigma$, let $C_\alpha$ be a closed subset of $\alpha$ with $\sup(C_\alpha)=\sup(\alpha)$ and $\min(C_\alpha)\ge\cf(\alpha)=\otp(C_\alpha)$;

$\br$ For every regular uncountable cardinal $\alpha<\kappa$,
set $C_\alpha:=\{ \gamma<\alpha\mid \forall \eta<\gamma[d_1(\eta,\alpha)<\gamma]\}$.

Note that it follows that $C_{\beta+1}=\{\beta\}$ for all $\beta<\kappa$.
Now, towards a contradiction, suppose that $\chi(\kappa)<\chi$. 
In particular, $\chi(\kappa)<\kappa$ so that, by \cite[Lemma~2.21(1)]{paper35}, 
$\kappa$ is a Mahlo cardinal.
\begin{claim} There exist $A \in [\kappa]^\kappa$, $\sigma<\chi$ and a pairwise disjoint subfamily $\mathcal B\s[\reg(\kappa)]^\sigma$ of size $\kappa$ with the property that, for every $\epsilon<\kappa$,
for some $b\in\mathcal B$, $A\cap\epsilon\s\bigcup_{\alpha\in b}C_\alpha$.
\end{claim}
\begin{proof} Set $\sigma:=\chi(\vec C)$, so that $\sigma\le\chi(\kappa)<\chi$.
Fix  $\Delta \in [\kappa]^\kappa$ and a sequence $\langle a_\epsilon\mid \epsilon<\kappa\rangle$ of sets in $[\kappa]^\sigma$ with the property that, for every $\epsilon<\kappa$,
$\Delta\cap\epsilon\s\bigcup_{\alpha\in a_\epsilon}C_\alpha$. 
Define a function $f_0:\reg(\kappa)\setminus(\sigma+1)\rightarrow\kappa$ via:
$$f_0(\epsilon):=\sup((a_\epsilon\cap\epsilon)\cup\bigcup\{ \otp(C_\alpha\cap\epsilon)\mid \alpha\in a_\epsilon\cap\Sigma\}).$$

Note that $f_0$ is regressive, because otherwise for some $\alpha\in a_\epsilon\cap\Sigma$,
$\otp(C_\alpha\cap\epsilon)=\epsilon>0$, and in particular $\epsilon>\min(C_\alpha)\ge\otp(C_\alpha)\ge\otp(C_\alpha\cap\epsilon)=\epsilon$.

Fix a stationary subset $S_0\s\dom(f_0)$ on which $f_0$ is constant with value, say, $\tau_0$. 
Define a function $f_1:S_0\setminus(\tau_0+1)\rightarrow\kappa$ via:
$$f_1(\epsilon):=\sup(\bigcup\{ \sup(C_\alpha\cap\epsilon)\mid \alpha\in a_\epsilon\cap\Sigma\}).$$

Note that $f_1$ is regressive, because otherwise for some $\alpha\in a_\epsilon\cap\Sigma$,
$\sup(C_\alpha\cap\epsilon)=\epsilon$, and in particular $\tau_0=f_0(\epsilon)\ge\otp(C_\alpha\cap\epsilon)\ge\cf(\epsilon)=\epsilon>\tau_0$.

Fix a stationary subset $S_1\s\dom(f_1)$ on which $f_1$ is constant with value, say, $\tau_1$. 
Set $A:=\Delta\setminus(\tau_0+\tau_1+1)$, and for every $\epsilon\in S_1$, set $b_\epsilon:=(a_\epsilon\cap\reg(\kappa))\setminus\epsilon$.

As $\min(b_\epsilon)\ge\epsilon$ for all $\epsilon\in S_1$, we may find $S_2\in[S_1]^\kappa$ such that $\mathcal B:=\{ b_\epsilon\mid \epsilon\in S_2\}$ 
is a pairwise disjoint family (of size $\kappa$).
Now, to see that $A$, $\sigma$ and $\mathcal B$ are as sought,
it suffices to show that for every $\epsilon\in S_2$, we have $A\cap\epsilon\s\bigcup_{\alpha\in b_\epsilon}C_\alpha$.

Let $\epsilon\in S_2$ and $\delta\in A\cap\epsilon$ be arbitrary.
In particular, $\delta\in \Delta\cap\epsilon$,
so we may fix $\alpha\in a_\epsilon$ such that $\delta\in C_\alpha$.

$\br$ If $\alpha\in a_\epsilon\cap\epsilon$, then $\alpha\le f_0(\epsilon)=\tau_0<\min(A)\le\delta$, contradicting the fact that $\delta\in C_\alpha\s\alpha$

$\br$ If $\alpha\in a_\epsilon\cap\Sigma$, then $\delta\le f_1(\epsilon)\le\tau_1<\min(A)\le\delta$ which is a contradiction.

So, $\alpha\in b_\epsilon$. Altogether, $A\cap\epsilon\s\bigcup_{\alpha\in b_\epsilon}C_\alpha$, as sought.
\end{proof}

Let $A$, $\sigma$ and $\mathcal B$ be given by the claim.
By throwing away at most one set from $\mathcal B$, we may assume that $\omega\notin b$ for all $b\in\mathcal B$.
By thinning out even further we may assume the existence of a club $D\s\kappa$ such that, for all $\delta\in D$ and $b\in\mathcal B$,
if $\min(b)>\delta$, then $A\cap\delta\s\bigcup_{\alpha\in b}C_\alpha$.

Now, by the choice of the coloring $d_1$, we may fix a stationary subset $\Delta\s\kappa$ and an ordinal $\eta<\kappa$
such that, for every $\delta\in\Delta$, there exists $b\in\mathcal B$ with $\min(b)>\max\{\eta,\delta\}$ satisfying $d_1[\{\eta\}\times b]=\{\delta\}$.

Fix $\delta\in D\cap\acc^+(A)\cap \Delta$,
and then fix $b\in\mathcal B$ with $\min(b)>\max\{\eta,\delta\}$ satisfying $d_1[\{\eta\}\times b]=\{\delta\}$.

As $\delta\in D$, $A\cap\delta\s\bigcup_{\alpha\in b}C_\alpha$.
As $\delta\in\acc^+(A)$, we may fix $\alpha\in b$ and $\gamma\in C_\alpha$ with $\eta<\gamma<\delta$.
As $\alpha\in b$, it is a regular uncountable cardinal, so it follows from the definition of $C_\alpha$ that $d_1(\eta,\alpha)<\gamma<\delta$,
contradicting the fact that $d_1(\eta,\alpha)=\delta$.
\end{proof}

\begin{q} Does  $\pl_1(\kappa,1,\chi(\kappa))$ hold for every inaccessible cardinal $\kappa$?
\end{q}

\section{From colorings to transformations}\label{pumpup}

For the sake of this section, we introduce the following ad-hoc principle.

\begin{defn} $\pr_1^+(\kappa,\theta,\chi)$ asserts 
the existence of a coloring $o:[\kappa]^2\rightarrow\theta$ satisfying:
\begin{enumerate}
\item For all nonzero $\alpha<\beta<\kappa$, $o(\alpha,\beta)<\alpha$;
\item For all $\zeta<\theta$, $\sigma<\chi$,
and a pairwise disjoint subfamily $\mathcal A\s [\kappa]^{\sigma}$ of size $\kappa$,
there is $\gamma<\kappa$ such that, for every $b\in[\kappa\setminus\gamma]^{\sigma}$,
for some $a\in\mathcal A\cap\mathcal P(\gamma)$, $o[a\times b]=\{\zeta\}$.
\end{enumerate}
\end{defn}

It turns out that the above variation is not much stronger than the original.

\begin{lemma}\label{Lemma35} Suppose that $\pr_1(\kappa,\kappa,\theta,\chi+\chi)$ holds.
Then so does $\pr_1^+(\kappa,\theta,\chi)$.
\end{lemma}
\begin{proof} Let $c$ be a witness to $\pr_1(\kappa,\kappa,\theta,\chi+\chi)$.
Define a coloring $o:[\kappa]^2\rightarrow\theta$
by letting $o(\alpha,\beta):=c(\alpha,\beta)$ whenever $c(\alpha,\beta)<\alpha<\beta<\kappa$,
and $o(\alpha,\beta):=0$, otherwise.
Towards a contradiction, suppose that $o$ is not as sought.
Let $\zeta$, $\sigma$ and $\mathcal A$ form together a counterexample.
This means that, for each $\gamma<\kappa$, we may fix $b_\gamma\in[\kappa\setminus\gamma]^{\sigma}$ 
such that, for all $a\in\mathcal A\cap\mathcal P(\gamma)$, $o[a\times b_\gamma]\neq\{\zeta\}$.
Also, fix $a_\gamma\in\mathcal A$ with $\min(a_\gamma)>\gamma$, and then set $x_\gamma:=a_\gamma\cup b_\gamma$.

Fix a club $C\s\kappa$ such that, for all $\gamma\in C$, $(\bigcup_{\gamma'<\gamma}x_{\gamma'})\s \gamma$.
In particular, $\langle x_\gamma\mid\gamma\in C\setminus\zeta\rangle$
is a $<$-increasing sequence of elements of $[\kappa]^{\le\sigma+\sigma}\s[\kappa]^{<\chi+\chi}$.
Thus, by the choice of $c$, we may find $(\gamma',\gamma)\in[\Gamma]^2$ such that $c[x_{\gamma'}\times x_\gamma]=\{\zeta\}$.
As $\min(a_{\gamma'})>\gamma'\ge\zeta$, it thus follows that $o[a_{\gamma'}\times b_\gamma]=\{\zeta\}$,
contradicting the fact that $a_{\gamma'}\in\mathcal A\cap\mathcal P(\gamma)$ and the choice of $b_\gamma$.
\end{proof}

\begin{thm}\label{thm42a} 
Suppose that $\mu$ is an infinite regular cardinal, $\chi\le\mu$,
and $\Gamma\s \mu^+$ is a nonreflecting stationary set.

If $\pr_1^+(\mu^+,\mu,\chi)$ holds, then so does $\pl_2(\mu^+,\Gamma,\chi)$.
\end{thm}
\begin{proof} 
Suppose that $\pr_1^+(\mu^+,\mu,\chi)$ holds,
as witnessed by $o:[\mu^+]^2\rightarrow\mu$.
Fix a $C$-sequence $\vec C=\langle C_\alpha\mid\alpha<\mu^+\rangle$ such that, for all $\alpha<\mu^+$,
$\otp(C_\alpha)=\cf(\alpha)$ and $C_\alpha\cap\Gamma=\emptyset$.
We shall walk along $\vec C$.
Fix a bijection $\pi: \mu\leftrightarrow\mu\times \mu$. For every $\beta<\mu^+$, fix a surjection $\psi_\beta:\mu\rightarrow\beta+1$.
Fix an almost disjoint family $\{Z_\epsilon \mid \epsilon<\mu^+\}\s[\mu]^\mu$.
For every ordinal $\xi<\mu$ and a pair $(\alpha,\beta)\in[\mu^+]^2$, we let
$$\xi^{\alpha,\beta}:=\min\{ n<\omega\mid \xi\in Z_{\Tr(\alpha,\beta)(n)}\text{ or }n=\rho_2(\alpha,\beta)+1\}.$$

Define $\mathbf t: [\mu^+]^2\rightarrow[\mu^+]^2$, as follows.
Let $\mathbf t(\alpha,\beta):=(\alpha^*,\beta^*)$ provided that the following hold:
\begin{itemize}
\item $(\tau,\xi):=\pi(o(\alpha,\beta))$;
\item $\beta^*:=\Tr(\alpha,\beta)(\xi^{\alpha,\beta})$ is $>\alpha$;
\item $\alpha^*:=\psi_{\beta^*}(\tau)$ is $<\alpha$.
\end{itemize}
Otherwise, just let $\mathbf t(\alpha,\beta):=(\alpha,\beta)$ 

To verify that $\mathbf t$ witnesses $\pl_2(\mu^+,\Gamma,\chi)$,
suppose that we are given $\sigma<\chi$ and a pairwise disjoint subfamily $\mathcal A\s[\mu^+]^\sigma$ of size $\mu^+$.
Fix a sequence  $\vec x=\langle x_\delta\mid \delta<\mu^+\rangle$ such that,
for all $\delta<\mu^+$, $x_\delta\in\mathcal A$ with $\min(x_\delta)>\delta$.

\begin{claim} There exists $\eta<\mu^+$ and a stationary $\Delta\s E^{\mu^+}_\mu$ such that,
for every $\delta\in\Delta$ and $\beta\in x_\delta$, $\lambda(\delta,\beta)\le\eta$.
\end{claim}
\begin{proof} 
Let $\delta\in E^{\mu^+}_\mu$ be arbitrary.
As $|x_\delta|<\chi\le\mu$, Fact~\ref{fact1} entails the existence of a large enough $\eta<\delta$ such that $\lambda(\delta,\beta)\le\eta$ for all $\beta\in x_\delta$.
Now, appeal to Fodor's lemma.
\end{proof}

Let $\eta$ and $\Delta$ be given by the claim. 
Let $D\s\mu^+$ be the club of all $\delta<\mu^+$ for which
there exists an elementary submodel $\mathcal M_\delta\prec\mathcal H_{\mu^{++}}$ with $\mathcal M_\delta\cap\mu^+=\delta$
such that $\{\langle x_\gamma\mid \gamma\in\Delta\rangle,o,\pi,\mu,\eta\}\in\mathcal M_\delta$.

\begin{claim} Let $(\alpha^*,\beta^*)\in [D\cap\Gamma]^2$. Then there exists $(a,b)\in\mathcal[A]^2$ such that $\mathbf t[a\times b]=\{(\alpha^*,\beta^*)\}$.
\end{claim}
\begin{proof} Fix $\delta^*\in\Delta$ above $\beta^*$. Pick $\xi\in Z_{\beta^*}\setminus\bigcup\{ Z_{\tr(\beta^*,\beta)(n)}\mid \beta\in x_{\delta^*}, n<\rho_2(\beta^*,\beta)\}$,
and $\tau<\mu$ such that $\psi_{\beta^*}(\tau)=\alpha^*$.
Let $\zeta:=\pi^{-1}(\tau,\xi)$.
As $\beta^*\in\Gamma$, it follows from Fact~\ref{fact1} that $\lambda(\beta^*,\delta^*)<\beta^*$.
Set $\eta^*:=\max\{\eta,\lambda(\beta^*,\delta^*)\}$. 
Let $\mathcal A':=\{x_\delta\mid \delta\in\Delta\setminus\eta^*\}$.
As $\mathcal A'$ and $o$ are in $\mathcal M_{\beta^*}$
there exists some $\gamma<\beta^*$ such that, for every $b\in[\mu^+\setminus\gamma]^{\sigma}$,
for some $a\in\mathcal A'\cap\mathcal P(\gamma)$, $o[a\times b]=\{\zeta\}$.
In particular, we may find $\delta\in\mathcal M_{\beta^*}\cap \Delta\setminus\eta^*$ such that $o[x_\delta\times x_{\delta^*}]=\{\zeta\}$.
Denote $a:=x_\delta$ and $b:=x_{\delta^*}$, so that $(a,b)\in[\mathcal A]^2$.
Let $(\alpha,\beta)\in a\times b$ be arbitrary.
Evidently, $$\max\{\lambda(\delta^*,\beta),\lambda(\beta^*,\delta^*)\}\le\eta^*\le\delta<\alpha<\beta^*<\delta^*<\beta.$$
So, Fact~\ref{fact2} implies that $\tr(\alpha,\beta)=\tr(\delta^*,\beta)\conc\tr(\beta^*,\delta^*)\conc\tr(\alpha,\beta^*)$.
Now, by the choice of $\xi$, we have $\Tr(\alpha,\beta)(\xi^{\alpha,\beta})=\beta^*$.
So, as $(\tau,\xi):=\pi(o(\alpha,\beta))$
and $\psi_{\beta^*}(\tau)=\alpha^*$, it follows that $\mathbf (\alpha,\beta)=(\alpha^*,\beta^*)$, as sought.
\end{proof}
This completes the proof.
\end{proof}

\begin{cor}\label{cor43} \begin{enumerate}
\item  For every integer $n\ge2$, $\pl_2(\aleph_1,\aleph_1,n)$ holds;
\item  If $\pr_1(\aleph_1,\aleph_1,\aleph_0,\aleph_0)$ holds, then so does $\pl_2(\aleph_1,\aleph_1,\aleph_0)$;
\end{enumerate}
\end{cor}
\begin{proof} 
(1) By Lemma~\ref{Lemma35}, Theorem~\ref{thm42a}, and the fact that, by Peng and Wu \cite{MR3742590}, $\pr_1(\aleph_1,\aleph_1,\aleph_1,n+n)$ holds for every positive integer $n$.

(2) By Lemma~\ref{Lemma35} and Theorem~\ref{thm42a}.
\end{proof}

\begin{thm}\label{thm45}
Suppose that $\chi\le\kappa$,
and $\Gamma\s \kappa$ is a nonreflecting stationary set
such that $\Gamma\cap E^\kappa_{>\sigma}$ stationary for every $\sigma<\chi$.

If $\pr_1^+(\kappa,\kappa,\chi)$ holds, then so does $\pl_2(\kappa,\Gamma,\chi)$.
\end{thm}
\begin{proof} 
By Lemma~\ref{lemma45}, we may fix a $C$-sequence $\vec C=\langle C_\alpha\mid\alpha<\kappa\rangle$ 
that avoids $\Gamma$ and satisfying $\chi_1(\vec C)\ge\chi$.
Fix a coloring $o:[\kappa]^2\rightarrow\kappa$, witnessing $\pr_1^+(\kappa,\kappa,\chi)$.
Fix a bijection $\pi:\kappa\leftrightarrow\omega\times\kappa\times\kappa$.
Define $\mathbf t: [\kappa]^2\rightarrow[\kappa]^2$, as follows.
Let $\mathbf t(\alpha,\beta):=(\alpha^*,\beta^*)$ provided that the following hold:
\begin{itemize}
\item $(n,\tau,\eta):=\pi(o(\alpha,\beta))$;
\item $\beta^*:=\Tr(\alpha,\beta)(\eta_{\alpha,\beta}+n)$ is $>\alpha$;
\item $\alpha^*:=\tau$ is $<\alpha$.
\end{itemize}
Otherwise, just let $\mathbf t(\alpha,\beta):=(\alpha,\beta)$.

To verify that $\mathbf t$ witnesses $\pl_2(\kappa,\Gamma,\chi)$,
suppose that we are given $\sigma<\chi$ and a pairwise disjoint subfamily $\mathcal A\s[\kappa]^\sigma$ of size $\kappa$.
As $\sigma<\chi_1(\vec C)$,
we may fix $\eta<\kappa$ and a sequence $\langle x_\delta\mid\delta\in\Delta\rangle$ such that $\Delta$ is a stationary subset of $\kappa$,
and, for every $\delta\in\Delta$, $x_\delta\in\mathcal A$ with $\min(x_\delta)>\delta$,
and, for every $\beta\in x_\delta$, $\lambda(\delta,\beta)=\eta$ and $\rho_2(\delta,\beta)=\eta_{\delta,\beta}$.

Let $D\s\kappa$ be the club of all $\delta<\kappa$ for which
there exists an elementary submodel $\mathcal M_\delta\prec\mathcal H_{\mu^{++}}$ with $\mathcal M_\delta\cap\kappa=\delta$
such that $\{\langle x_\gamma\mid \gamma\in\Delta\rangle,o,\pi,\eta\}\in\mathcal M_\delta$.

\begin{claim} Let $(\alpha^*,\beta^*)\in [D\cap\Gamma]^2$. Then there exists $(a,b)\in\mathcal[A]^2$ such that $\mathbf t[a\times b]=\{(\alpha^*,\beta^*)\}$.
\end{claim}
\begin{proof} Fix $\delta^*\in\Delta$ above $\beta^*$. Let $n:=\rho_2(\beta^*,\delta^*)$.
Let $\zeta:=\pi^{-1}(n,\alpha^*,\eta)$.
As $\beta^*\in\Gamma$, it follows from Fact~\ref{fact1} that $\lambda(\beta^*,\delta^*)<\beta^*$.
Set $\eta^*:=\max\{\eta,\lambda(\beta^*,\delta^*)\}$. 
Let $\mathcal A':=\{x_\delta\mid \delta\in\Delta\setminus\eta^*\}$.
As $\mathcal A'$ and $o$ are in $\mathcal M_{\beta^*}$
there exists some $\gamma<\beta^*$ such that, for every $b\in[\kappa\setminus\gamma]^{\sigma}$,
for some $a\in\mathcal A'\cap\mathcal P(\gamma)$, $o[a\times b]=\{\zeta\}$.
In particular, we may find $\delta\in\mathcal M_{\beta^*}\cap \Delta\setminus\eta^*$ such that $o[x_\delta\times x_{\delta^*}]=\{\zeta\}$.
Denote $a:=x_\delta$ and $b:=x_{\delta^*}$, so that $(a,b)\in[\mathcal A]^2$.
Let $(\alpha,\beta)\in a\times b$ be arbitrary.
Evidently, $$\max\{\lambda(\delta^*,\beta),\lambda(\beta^*,\delta^*)\}\le\eta^*\le\delta<\alpha<\beta^*<\delta^*<\beta.$$
So, Fact~\ref{fact2} implies that $\tr(\alpha,\beta)=\tr(\delta^*,\beta)\conc\tr(\beta^*,\delta^*)\conc\tr(\alpha,\beta^*)$.
Now, by the choice of $\eta$, we have $\Tr(\alpha,\beta)(\eta_{\alpha,\beta})=\delta^*$ and $\Tr(\alpha,\beta)(\eta_{\alpha,\beta}+n)=\beta^*$.
So, as $(n,\alpha^*,\eta):=\pi(o(\alpha,\beta))$,
it altogether follows that $\mathbf (\alpha,\beta)=(\alpha^*,\beta^*)$, as sought.
\end{proof}
This completes the proof.
\end{proof}

\begin{cor}
Suppose that $\chi<\chi^+<\kappa$ are infinite regular cardinals,
and $\Gamma\s\kappa$ is a nonreflecting stationary set
such that $\Gamma\cap E^\kappa_{\ge\chi}$ stationary.

Then $\pl_2(\kappa,\Gamma,\chi)$ holds.
\end{cor}
\begin{proof}
By the main result of \cite{paper15},
the hypothesis implies that $\pr_1(\kappa,\kappa,\kappa,\chi)$ holds.
So by Lemma~\ref{Lemma35}, $\pr_1^+(\kappa,\kappa,\chi)$ holds, as well.
In addition, by Lemma~\ref{lemma45}, we may find a $C$-sequence $\vec C$ avoiding $\Gamma$ with $\chi_1(\vec C)\ge\chi$.
Now, appeal to Theorem~\ref{thm45}.
\end{proof}

\begin{thm}\label{thm47} Suppose that $\pr_1^+(\kappa,\kappa,\chi)$ holds for some cardinal $\chi\le\kappa$,
and that $\Gamma\s \kappa$ is a stationary set.
\begin{enumerate}
\item If there is a $C$-sequence $\vec C$ over $\kappa$ with $\chi_1(\vec C)\ge\chi$, 
then $\pl_1(\kappa,\kappa,\chi)$ holds;
\item If there is a $C$-sequence $\vec C$ over $\kappa$ 
with $\chi_2(\vec C,\Gamma)\ge\chi$, then $\pl_2(\kappa,\Gamma,\chi)$ holds.
\end{enumerate}
\end{thm}
\begin{proof} We focus on the proof of Clause~(2). The proof of Clause~(1) is very similar.
Fix a coloring $o:[\kappa]^2\rightarrow\kappa$ witnessing $\pr_1^+(\kappa,\kappa,\chi)$.
Suppose that $\vec C=\langle C_\alpha\mid\alpha<\kappa\rangle$ is a $C$-sequence over $\kappa$
with $\chi_2(\vec C,\Gamma)\ge\chi$, and let us walk along $\vec C$.
Fix a bijection $\pi:\kappa\leftrightarrow\kappa\times\kappa$.
Define $\mathbf t: [\kappa]^2\rightarrow[\kappa]^2$, as follows.
Let $\mathbf t(\alpha,\beta):=(\alpha^*,\beta^*)$ provided that the following hold:
\begin{itemize}
\item $(\tau,\eta):=\pi(o(\alpha,\beta))$;
\item $\beta^*:=\Tr(\alpha,\beta)(\eta_{\alpha,\beta})$ is $>\alpha$;
\item $\alpha^*:=\tau$ is $<\alpha$.
\end{itemize}
Otherwise, just let $\mathbf t(\alpha,\beta):=(\alpha,\beta)$.

To verify that $\mathbf t$ is as sought,
suppose that we are given $\sigma<\chi$ and a pairwise disjoint subfamily $\mathcal A\s[\kappa]^\sigma$ of size $\kappa$.
As $\sigma<\chi_2(\vec C,\Gamma)$, fix $\Delta\s\Gamma$ for which $\Gamma\setminus\Delta$ is nonstationary
such that, 
for every $\delta\in\Gamma\cap C$,
there exist an ordinal $\eta_\delta<\delta$ and a subfamily $\mathcal A_\delta\in[\mathcal A]^\kappa$
such that,
for all $b\in\mathcal A_\delta$ and $\beta\in b$,
$\lambda(\delta,\beta)=\eta_\delta$ and $\rho_2(\delta,\beta)=(\eta_\delta)_{\delta,\beta}$.

Next, let $D$ be the club of all $\delta<\kappa$ for which
there exists an elementary submodel $\mathcal M_\delta\prec\mathcal H_{\mu^{++}}$ with $\mathcal M_\delta\cap\kappa=\delta$
such that $\{\langle \mathcal A_\gamma\mid \gamma\in\Delta\rangle,o,\pi\}\in\mathcal M_\delta$.

\begin{claim} Let $(\alpha^*,\beta^*)\in [\Delta\cap D]^2$. Then there exists $(a,b)\in\mathcal[A]^2$ such that $\mathbf t[a\times b]=\{(\alpha^*,\beta^*)\}$.
\end{claim}
\begin{proof} Denote $\eta:=\eta_{\beta^*}$. 
Evidently, $\{ a\in \mathcal A_{\alpha^*}\mid \min(a)>\eta\}$ and  $\zeta:=\pi^{-1}(\alpha^*,\eta)$ are in $\mathcal M_{\beta^*}$.
As $o\in\mathcal M_{\beta^*}$, it follows that there exists $\gamma<\beta^*$
such that, for every $b\in[\kappa\setminus\gamma]^{\sigma}$,
for some $a\in\mathcal A_{\alpha^*}\cap\mathcal P(\gamma)$ with $\min(a)>\eta$, $o[a\times b]=\{\zeta\}$.

Fix an arbitrary $b\in\mathcal A_{\beta^*}$ with $\min(b)>\beta^*$,
and then pick $a\in\mathcal A_{\alpha^*}\cap\mathcal P(\gamma)$ with $\min(a)>\eta$ such that $o[a\times b]=\{\zeta\}$.
Now, let $(\alpha,\beta)\in a\times b$, and we shall show that $\mathbf t(\alpha,\beta)=(\alpha^*,\beta^*)$.
All of the following hold:
\begin{itemize}
\item $\eta<\alpha<\gamma<\beta^*<\beta$,
\item $\lambda(\beta^*,\beta)=\eta$, and
\item  $\rho_2(\beta^*,\beta)=\eta_{\beta^*,\beta}$.
\end{itemize}

So, by Lemma~\ref{Lemma212} (using $\delta:=\beta^*$),
$\tr(\alpha,\beta)(\eta_{\alpha,\beta})=\beta^*$.
Recalling that $\pi(o(\alpha,\beta))=(\alpha^*,\eta)$, we infer from the definition of $\mathbf t$ that $\mathbf t(\alpha,\beta)=(\alpha^*,\beta^*)$,
as sought.
\end{proof}
This completes the proof.
\end{proof}

\begin{cor}\label{cor48}
Suppose that $\chi\le\kappa$ is an infinite cardinal such that $\pr_1(\kappa,\kappa,\kappa,\chi)$ holds.
\begin{enumerate}
\item If $\square(\kappa)$ holds, then so does $\pl_1(\kappa,\kappa,\chi)$;
\item If $\boxtimes^-(\kappa)$ holds, then so does $\pl_2(\kappa,\kappa,\chi)$.
\end{enumerate}
\end{cor}
\begin{proof} By Corollary~\ref{cor43}, we may assume that $\kappa\ge\aleph_2$.

(1) If $\square(\kappa)$ holds, then by Lemma~\ref{Lemma32}(1),
we may find a $C$-sequence $\vec C$ over $\kappa$ such that $\chi_1(\vec C)=\sup(\reg(\kappa))$.
Now, appeal to Theorem~\ref{thm47}.

(2) If $\boxtimes^-(\kappa)$ holds, then by Lemma~\ref{Lemma32}(2),
we may find a $C$-sequence $\vec C$ over $\kappa$ such that $\chi_2(\vec C,\kappa)=\sup(\reg(\kappa))$.
Now, appeal to Theorem~\ref{thm47}.
\end{proof}
\begin{remark} By \cite[Proposition~2.19(1)]{paper44},
it is consistent that for an inaccessible cardinal $\kappa$,
$\pr_1(\kappa,\kappa,\kappa,\omega)$
holds, but $\pl_1(\kappa,1,\omega)$ fails.
\end{remark}

\section{inaccessible cardinals}\label{inaccessibles}

\begin{fact}[\cite{paper47}]\label{factfrom47} Assume any of the following:
\begin{itemize}
\item $\diamondsuit(S)$ holds for some stationary $S\s\kappa$ that does not reflect at regulars;
\item $\diamondsuit^*(\kappa)$ holds.
\end{itemize}
Then there exists a coloring $d_0:[\kappa]^2\rightarrow\kappa$ satisfying that, for every stationary $\Delta\s\kappa$,
there exists $\tau<\kappa$ such that $d_0[\{\tau\}\circledast\Delta]=\kappa$.
\end{fact}

\begin{thm}\label{thm52}  Suppose that $\kappa=\kappa^{<\kappa}$ is an inaccessible cardinal and there 
are a map $d_0:[\kappa]^2\rightarrow\kappa$ as in Fact~\ref{factfrom47}
and a map $d_1:[\kappa]^2\rightarrow\kappa$ good for $\chi$ in the sense of Lemma~\ref{done}.
Then $\pro{\kappa}{\kappa}{1}{\kappa}{\chi}$ holds.
\end{thm}
\begin{proof} Fix $d_0$ and $d_1$ as above. Fix a bijection $\pi:\kappa\leftrightarrow\kappa\times\kappa$. 
Fix a surjection $\psi:\kappa\rightarrow\kappa$ such that the preimage of any singleton is cofinal in $\kappa$.
Next, define an auxiliary coloring $e:[\kappa]^2\rightarrow\kappa$ as follows.
Given $j<\beta<\kappa$, set $(\tau,\eta):=\pi(\psi(j))$ and then let $e(j,\beta):=\psi(d_0(\tau,d_1(\eta,\beta)))$.

\begin{claim}\label{claim431a} For every $\sigma<\chi$ and every pairwise disjoint subfamily $\mathcal B\s[\kappa]^{\sigma}$ of size $\kappa$, there are cofinally many $j<\kappa$ such that,
for every $\gamma<\kappa$, there are $\kappa$ many $b\in\mathcal B$ with $e[\{j\}\times b]=\{\gamma\}$.
\end{claim}
\begin{proof} Let $\mathcal B\s[\kappa]^{\sigma}$ be as above.
As $d_1$ is good for $\chi$, we may fix a stationary set $\Delta\s\kappa$ and an ordinal $\eta<\kappa$
such that for every $\delta\in\Delta$, there exists $b\in\mathcal B$ with $\min(b)>\max\{\eta,\delta\}$ satisfying $d_1[\{\eta\}\times b]=\{\delta\}$. 
Now, by the choice of $d_0$, find $\tau<\kappa$ such that $d_0[\{\tau\}\circledast\Delta]=\kappa$.
Evidently, $J:=\{ j<\kappa\mid \psi(j)=\pi^{-1}(\tau,\eta)\}$ is cofinal in $\kappa$.
Let $j\in J$ be arbitrary.  

Now, given $\gamma<\kappa$, 
as $d_0[\{\tau\}\circledast\Delta]=\kappa$, the following set has size $\kappa$:
$$\Delta':=\{\delta\in\Delta\setminus(\tau+1)\mid \psi(d_0(\tau,\delta))=\gamma\}.$$
Recalling the choice of $\eta$, it follows that the following set has size $\kappa$, as well:
$$\mathcal B':=\{ b\in\mathcal B\mid \exists \delta\in\Delta'\,(\min(b)>\max\{\eta,\delta,j\}\ \&\ d_1[\{\eta\}\times b]=\{\delta\})\}.$$

Let $b\in\mathcal B'$ be arbitrary. Let $\delta\in\Delta'$ be a witness for $b$ being in $\mathcal B'$.
For every $\beta\in b$, 
$e(j,\beta)=\psi(d_0(\tau,d_1(\eta,\beta)))=\psi(d_0(\tau,\delta))=\gamma$.
\end{proof}

Fix a strictly increasing sequence $\langle \kappa_j\mid j<\kappa\rangle$ of infinite cardinals below $\kappa$, such that,
for all $j<\kappa$, $(\sup_{i<j}\kappa_i)<\kappa_j$.
For every $j<\kappa$, let $\Phi_j:=\bigcup\{{}^x\kappa\mid x\s \kappa, |x|=\kappa_j\}$.
As $\kappa^{<\kappa}=\kappa$, $|\Phi_j|=\kappa$, so we may fix an injective enumeration
$\langle \phi_j^\gamma\mid \gamma<\kappa\rangle$ of $\Phi_j$.
Now, define a coloring $c:[\kappa]^2\rightarrow \kappa$ by letting for all
$\alpha<\beta<\kappa$:
$$c(\alpha,\beta):=\begin{cases}
0&\text{if }\alpha\notin\bigcup_{i<\kappa}\dom(\phi_i^{e(i,\beta)});\\
\phi_j^{e(j,\beta)}(\alpha)&\text{if }j=\min\{i<\kappa\mid \alpha\in\dom(\phi_i^{e(i,\beta)})\}.
\end{cases}$$

To see that $c$ is as sought, fix $\sigma<\chi$ and pairwise disjoint subfamilies $\mathcal A,\mathcal B$ of $[\kappa]^\sigma$ of size $\kappa$.
By Claim~\ref{claim431a}, fix $j<\kappa$ with $\kappa_j>\sigma$ such that, for every $\gamma<\kappa$, 
there are $\kappa$ many $b\in\mathcal B$  with $e[\{j\}\times b]=\{\gamma\}$.
Let $\langle a_\iota\mid \iota<\kappa_j\rangle$ be an injective sequence consisting of elements of $\mathcal A$.

\begin{claim} There exists $\iota<\kappa_j$ such that, for every
$\delta<\kappa$, there is $b\in\mathcal B$ with $a_\iota<b$ such that $c[a_\iota\times b]=\{\delta\}$.
\end{claim}
\begin{proof} Suppose not. Then, for every $\iota<\kappa_j$,
we may find some $\delta_\iota<\kappa$ such that, for all $b\in\mathcal B$ with $a_\iota<b$, $c[a_\iota\times b]\neq\{\delta_\iota\}$.
Define a function $\phi:\biguplus\{a_\iota\mid \iota<\kappa_j\}\rightarrow\kappa$ by letting
$\phi(\alpha):=\delta_\iota$ iff $\alpha\in a_\iota$.
As $\kappa_j>\sigma$, we infer that $\phi\in\Phi_j$,
so we may fix $\gamma<\kappa$ such that $\phi=\phi_j^\gamma$. 
Now, by the choice of $j$, let us pick $b\in\mathcal B$ with $\dom(\phi)<b$ such that $e[\{j\}\times b]=\{\gamma\}$.

For every $i<j$ and $\beta\in b$, let
$x_i^\beta:=\dom(\phi_i^{e(i,\beta)})$,
so that $|x_i^\beta|=\kappa_i$.
Next, set $x:=\bigcup\{ x_i^\beta\mid i<j,\beta\in b\}$,
so that $|x|<\kappa_j$. In particular, we may fix $\iota<\kappa_j$ such that $a_\iota\cap x=\emptyset$.
Now, let $(\alpha,\beta)\in a_\iota\times b$ be arbitrary.
As $e(j,\beta)=\gamma$, we infer that $\phi_{j}^{e(j,\beta)}=\phi$.
In particular, $\alpha\in a_\iota\s \dom(\phi_{j}^{e(j,\beta)})$.
Recalling that $\alpha\notin x$, it follows that $\min\{i<\kappa\mid \alpha\in\dom(\phi_i^{e(i,\beta)})\}=j$,
and hence
$$c(\alpha,\beta)=\phi_{j}^{e(j,\beta)}(\alpha)=\phi(\alpha)=\delta_\iota.$$
Altogether, $c[a_\iota\times b]=\{\delta_\iota\}$, contradicting the choice of $\delta_\iota$.
\end{proof}
This completes the proof.
\end{proof}

\begin{thm}\label{thm54} Suppose that $\kappa$ is an inaccessible cardinal and $\boxtimes^-(\kappa)$ and $\diamondsuit(\kappa)$ both hold.
Then $\pro{\kappa}{\kappa}{1}{\kappa}{\kappa}$ holds, as well.
\end{thm}
\begin{proof} As $\diamondsuit(\kappa)$ holds, let us fix a sequence $\langle f_\delta\mid\delta<\kappa\rangle$ such that,
for every $\delta<\kappa$, $f_\delta$ is a function from to $\delta$ to $\delta$,
and, for every function $f:\kappa\rightarrow\kappa$, the set $G(f):=\{\delta<\kappa\mid f\restriction\delta=f_\delta\}$ is stationary.
Let $\vec C$ be a $\boxtimes^-(\kappa)$-sequence, and we shall walk along $\vec C$.
Now, pick any coloring $d:[\kappa]^2\rightarrow\kappa$ such that, for every $\eta,\beta<\kappa$ with $\eta+1<\beta$,
$d(\eta,\beta)=f_{\Tr(\eta+1,\beta)(\eta_{\eta+1,\beta})}(\eta)$.

\begin{claim} Let $\sigma<\kappa$ and let $\mathcal B$ 
be a pairwise disjoint subfamily $[\kappa]^{\sigma}$ of size $\kappa$.
Then there exists $\eta<\kappa$ such that,
for every $\gamma<\kappa$, there are $\kappa$ many $b\in\mathcal B$ with $d[\{\eta\}\times b]=\{\gamma\}$.
\end{claim}
\begin{proof} Suppose not.
Fix a function $f:\kappa\rightarrow\kappa$ such that, for every $\eta<\kappa$,
$$\mathcal B_\eta:=\{ b\in\mathcal B\mid d[\{\eta\}\times b]=\{f(\eta)\}\}$$
has size $<\kappa$.

By Lemma~\ref{Lemma32}(2), $\chi_2(\vec C,\kappa)=\kappa>\sigma$,
so since $G(f)$ is stationary, we may fix $\delta\in G(f)$, $\eta<\delta$ and $b\in\mathcal B\setminus\mathcal B_\eta$
such that, for every $\beta\in b$, $\lambda(\delta,\beta)=\eta$ and $\rho_2(\delta,\beta)=\eta_{\delta,\beta}$.

For each $\beta\in b$, by appealing to Lemma~\ref{Lemma212} with $\alpha:=\eta+1$, we get that 
$$d(\eta,\beta)=f_{\Tr(\eta+1,\beta)(\eta_{\eta+1,\beta})}(\eta)=f_\delta(\eta)=f(\eta).$$
So $b\in\mathcal B_\eta$, contradicting its choice.
\end{proof}
Fix a surjection $\psi:\kappa\rightarrow\kappa$ such that the preimage of any singleton is cofinal in $\kappa$.
Define an auxiliary coloring $e:[\kappa]^2\rightarrow\kappa$ via $e(j,\beta):=d(\psi(j),\beta)$.
By the preceding claim, the proposition of Claim~\ref{claim431a} holds,
so we can continue as in the proof of Theorem~\ref{thm52}.
\end{proof}
\begin{remark} The preceding proof actually shows that if $\p^\bullet(\kappa,2,{\sq},1,\{\kappa\},2)$ holds
(see \cite[Definition~5.9 and Proposition~5.10]{paper23}), then so does $\pro{\kappa}{\kappa}{1}{\kappa}{\kappa}$.
\end{remark}

The next corollary yields Clauses (1)--(3) of Theorem~B.

\begin{cor} Suppose that $\kappa$ is an inaccessible cardinal.
\begin{enumerate}
\item If $\square(\kappa)$ and $\diamondsuit(S)$ both hold for some stationary $S\s\kappa$ that does not reflect at regulars,
then so does $\pl_1(\kappa,\kappa,\kappa)$;
\item If $\square(\kappa)$ and $\diamondsuit^*(\kappa)$ both hold, then so does $\pl_1(\kappa,\kappa,\kappa)$;
\item If $\boxtimes^-(\kappa)$ and $\diamondsuit(\kappa)$ both hold, then so does $\pl_2(\kappa,\kappa,\kappa)$.
\end{enumerate}
\end{cor}
\begin{proof} To prove Clauses (1) and (2), let $d_0$ be given by Fact~\ref{factfrom47}.
Next, let $d_1$ be given by Lemma~\ref{done}(2), using Lemma~\ref{Lemma32}(1). 
Now, appeal to Theorem~\ref{thm52}.

(3) By Theorem~\ref{thm54} and Corollary~\ref{cor48}.
\end{proof}

\section{Successors of regular cardinals}\label{galvin}

\begin{lemma}\label{grouping} Suppose that $\mu$ is an infinite regular cardinal.
Then there exists a sequence $\vec f=\langle f_j\mid j<\mu\rangle$ 
of functions from $\mu^+$ to $\mu^+$ such that,
for every pairwise disjoint subfamily $\mathcal B\s[\mu^+]^{<\mu}$ of size $\mu^+$, 
for every $\gamma<\mu^+$, there exist $j<\mu$ and $b\in\mathcal B$ such that $f_{j}[b]=\{\gamma\}$.
\end{lemma}
\begin{proof} Fix a surjection $g:\mu^+\rightarrow\mu^+$ such that the preimage of any singleton is cofinal in $\mu^+$.
As $\mu$ is regular, using \cite[Lemma~6.25]{TodWalks}, we may fix a function $p:[\mu^+]^2\rightarrow\mu$ having injective and $\mu$-coherent fibers; 
the latter means that $|\{ \alpha<\beta\mid p(\alpha,\beta)\neq p(\alpha,\beta')\}|<\mu$ for all $\beta<\beta'<\mu^+$.
Now, for every $j<\mu$, define a function $f_j:\mu^+\rightarrow\mu^+$ via:
$$f_j(\beta):=\begin{cases}
0,&\text{if }j\notin\{p(\alpha,\beta)\mid \alpha<\beta\};\\
g(\alpha)&\text{if }p(\alpha,\beta)=j.
\end{cases}$$

Let $\mathcal B$ be a pairwise disjoint subfamily of $[\mu^+]^{<\mu}$ of size $\mu^+$,
and let $\gamma<\mu^+$ be a prescribed color.
Find $\delta\in E^{\mu^+}_\mu$ such that $A:=\{\alpha<\delta\mid g(\alpha)=\gamma\}$ is cofinal in $\delta$.
Pick $b\in\mathcal B$ with $\min(b)>\delta$. As $p$ is $\mu$-coherent and $|b|<\mu$, we may find some $\Delta\in[\delta]^{<\mu}$ such that for all $\beta,\beta'\in b$ and $\alpha\in\delta\setminus\Delta$, $p(\alpha,\beta)=p(\alpha,\beta')$.
Now, pick $\alpha\in A\setminus\Delta$. Let $j$ denote the unique element of singleton $\{p(\alpha,\beta)\mid \beta\in b\}$.
Then, for all $\beta\in b$, $f_j(\beta)=g(\alpha)=\gamma$, as sought.
\end{proof}

\begin{defn}[\cite{MR511564}]\label{stickp} $\stick(\mu^+)$ asserts the existence of a sequence $\langle X_\gamma\mid \gamma<\mu^+\rangle$ 
such that, for every $X\in[\mu^+]^{\mu^+}$, there exists $\gamma<\mu^+$ such that $X_\gamma\in[X]^\mu$.
\end{defn}

\begin{thm}\label{thm62} Suppose that $\mu$ is a regular uncountable cardinal, and $\stick(\mu^+)$ holds.
Then $\pro{\mu^+}{\mu^+}{1}{\mu^+}{\mu}$ holds.

If $2^\mu=\mu^+$, then moreover $\pro{\mu^+}{\mu}{1}{\mu^+}{\mu}$ holds.
\end{thm}
\begin{proof} 
For an ordinal $\eta$, let $\Phi_\eta$ denote the collection of all sequences $\langle (a_\iota,\delta_\iota)\mid \iota<\eta\rangle$
such that $\langle a_\iota\mid \iota<\eta\rangle$ is a sequence of pairwise disjoint elements of $[\mu^+]^{<\mu}\setminus\{\emptyset\}$,
and $\langle \delta_\iota\mid \iota<\eta\rangle$ is a sequence of elements of $\mu^+$.
Define an ordering $\unlhd$ of $\bigcup_{\eta<\mu^+}\Phi_\eta$ by letting
$$\langle (a_\iota,\delta_\iota)\mid \iota<\eta\rangle\unlhd \langle (a_\iota',\delta_\iota')\mid \iota<\eta'\rangle$$
iff for every $\iota<\eta$, there exists $\iota'<\eta'$ such that $a_\iota\supseteq a_{\iota'}$ and $\delta_\iota=\delta_{\iota'}$.

\begin{claim} There exists a sequence
$\langle \phi_\gamma\mid \gamma<\mu^+\rangle$ of elements of $\Phi_\mu$ such that for every $\phi\in\Phi_{\mu^+}$,
there exists $\gamma<\mu^+$ with $\phi_\gamma\unlhd \phi$.
\end{claim}
\begin{proof} 
For every $\beta<\mu^+$, fix a surjection $\psi_\beta:\mu\rightarrow\beta+1$.
Then let $\mathcal A:=\{ \psi_\beta[\epsilon]\setminus\alpha\mid \epsilon<\mu, \alpha<\beta<\mu^+\}$.
Evidently, $|\mathcal A|=\mu^+$,
so, as $\stick(\mu^+)$ holds, we may fix a sequence $\langle X_\gamma\mid \gamma<\mu^+\rangle$ 
with the property that, for every $X\in[\mathcal A\times \mu^+]^{\mu^+}$, there exists $\gamma<\mu^+$ such that $X_\gamma\in[X]^\mu$.
Now, pick a sequence $\langle \phi_\gamma\mid \gamma<\mu^+\rangle$ of elements of $\Phi_\mu$ with the property that,
for every $\gamma<\mu^+$, if there exists $\phi\in\Phi_{\mu}$ such that $\im(\phi)\s X_\gamma$,
then $\phi_\gamma$ is such a $\phi$.

To see that $\langle \phi_\gamma\mid \gamma<\mu^+\rangle$ is as sought, 
let $\langle (a_\iota,\delta_\iota)\mid \iota<\mu^+\rangle$ be an arbitrary element of $\Phi_{\mu^+}$.
For every $\iota<\mu^+$, let $\alpha_\iota:=\min(a_\iota)$,  $\beta_\iota:=\sup(a_\iota)$,
and $\epsilon_\iota:=\ssup(\psi_{\beta_\iota}^{-1}[a_\iota])$.
Clearly, $\overline{a_\iota}:=\psi_{\beta_\iota}[\epsilon_\iota]\setminus\alpha_\iota$ is an element of $\mathcal A$
satisfying $\overline{a_\iota}\supseteq a_\iota$ and $\min(\overline{a_\iota})=\min(a_\iota)$.
Recalling that $\langle a_\iota\mid \iota<\mu^+\rangle$ is a sequence of pairwise disjoint elements of $[\mu^+]^{<\mu}\setminus\{\emptyset\}$,
it follows that we may fix a sparse enough $I\in[\mu^+]^{\mu^+}$ such that 
$\langle \overline{a_\iota}\mid \iota\in I\rangle$ is a $<$-increasing sequence of elements of $[\mu^+]^{<\mu}\setminus\{\emptyset\}$.
In effect,  $X:=\{ (\overline{a_\iota},\delta_\iota)\mid \iota\in I\}$
is in $[\mathcal A\times \mu^+]^{\mu^+}$.
Now, pick $\gamma<\mu^+$ such that $X_\gamma\in[X]^\mu$.
Evidently, $\phi_\gamma\unlhd\phi$.
\end{proof}

If $2^\mu=\mu^+$, then we fix an injective enumeration 
$\langle \phi_\gamma\mid \gamma<\mu^+\rangle$ of $\Phi_\mu$. Otherwise, we 
let $\langle \phi_\gamma\mid \gamma<\mu^+\rangle$ be given by the preceding claim.
Let $\vec f$ be given by Lemma~\ref{grouping}.
Fix a surjection $\psi:\mu\rightarrow\mu$ such that the preimage of any singleton is stationary. 
For every $\beta<\mu^+$ and $j<\mu$, write $\langle (a^{j,\beta}_\iota,\delta^{j,\beta}_\iota)\mid \iota<\mu\rangle$ for $\phi_{f_{\psi(j)}(\beta)}$. 

Let $\beta<\mu^+$. We now recursively construct a strictly increasing sequence 
$\langle \iota^{j,\beta}\mid j<\mu\rangle$ of ordinals below $\mu$.
Suppose that $j<\mu$ and that $\langle \iota^{i,\beta}\mid i<j\rangle$ has already been defined.
If there exists $\iota<\mu$ such that:
\begin{itemize}
\item $a_\iota^{j,\beta}\s\beta\setminus\bigcup_{i<j}a_{\iota^{i,\beta}}^{i,\beta} $, and
\item $\iota\ge \sup_{i<j}(\iota^{i,\beta}+1)$,
\end{itemize}
then let $\iota^{j,\beta}$ denote the least such $\iota$.
Otherwise, just let $\iota^{j,\beta}:=\sup_{i<j}(\iota^{i,\beta}+1)$.

Finally, we define a coloring $c:[\mu^+]^2\rightarrow \mu^+$ by letting for all $\alpha<\beta<\mu^+$:
$$c(\alpha,\beta):=\begin{cases}
0&\text{if }\alpha\notin\bigcup_{i<\mu}a_{\iota^{i,\beta}}^{i,\beta};\\
\delta^{j,\beta}_{\iota^{j,\beta}}&\text{if }j=\min\{i<\mu\mid \alpha\in a_{\iota^{i,\beta}}^{i,\beta}\}.
\end{cases}$$

Assuming $2^\mu=\mu^+$, to see that $c$ witnesses $\pro{\mu^+}{\mu}{1}{\mu^+}{\mu}$,
fix $\sigma<\mu$ and pairwise disjoint subfamilies $\mathcal A,\mathcal B$ of $[\mu^+]^\sigma$ such that $|\mathcal A|=\mu$ and $|\mathcal B|=\mu^+$.
Assuming $2^\mu>\mu^+$, to see that $c$ witnesses $\pro{\mu^+}{\mu^+}{1}{\mu^+}{\mu}$,
fix $\sigma<\mu$ and pairwise disjoint subfamilies $\mathcal A,\mathcal B$ of $[\mu^+]^\sigma$ such that $|\mathcal A|=|\mathcal B|=\mu^+$.

Towards a contradiction, suppose that, for every $a\in\mathcal A$, there exists some $\delta(a)<\mu^+$ such that,
for every $b\in\mathcal B$ with $a<b$, $c[a\times b]\neq\{\delta(a)\}$. 
Fix $\phi\in\Phi_{|\mathcal A|}$ such that $\im(\phi)=\{ (a,\delta(a))\mid a\in\mathcal A\}$.

$\br$ If $2^\mu=\mu^+$, then $\phi\in\Phi_\mu$, and we may fix $\gamma<\mu^+$ such that 
$\phi=\phi_\gamma$. In particular, $\phi_\gamma\unlhd \phi$.

$\br$ If $2^\mu>\mu^+$, then $\phi\in\Phi_{\mu^+}$,
so we may fix $\gamma<\mu^+$ with $\phi_\gamma\unlhd \phi$.

\smallskip

Write $\phi_\gamma$ as $\langle (a_\iota,\delta_\iota)\mid \iota<\mu\rangle$.
Set $\epsilon:=\ssup(\bigcup_{\iota<\mu}a_\iota)$,
and then fix a bijection $\pi:\epsilon\leftrightarrow\mu$.
\begin{claim}
\begin{enumerate}
\item $D:=\{ j<\mu\mid \{\iota<\mu\mid \pi[a_\iota]\cap j\neq\emptyset\}\s j\}$ is a club in $\mu$;
\item For every $\beta<\mu^+$, $C_\beta:=\left\{ j<\mu\Mid \pi\left[\bigcup\nolimits_{i<j}a_{\iota^{i,\beta}}^{i,\beta}\right]\s j=\sup_{i<j}(\iota^{i,\beta}+1)\right\}$ is a club in $\mu$.
\end{enumerate}
\end{claim}
\begin{proof} (1) Define a function $g_0:\mu\rightarrow\mu$ via $g_0(i):=\sup\{ \iota<\mu\mid i\in \pi[a_\iota]\}$.
As the elements of $\langle a_\iota\mid \iota<\mu\rangle$ are pairwise disjoint, $g_0$ is well-defined.
Clearly, $D$ coincides with the club $\{j<\mu\mid g_0[j]\s j\}$.

(2) Let $\beta<\mu^+$. It is clear that $C_\beta$ is closed.
To see it is unbounded, define two functions $g_1,g_2:\mu\rightarrow\mu$ via $g_1(i):=\sup(\pi[a_{\iota^{i,\beta}}^{i,\beta}])$
and $g_2(i):=\iota^{i,\beta}+1$.
As the elements of $\langle \pi[a_\iota]\mid \iota<\mu\rangle$ are elements of $[\mu]^{<\mu}$, $g_1$ is well-defined.
Recalling that the sequence $\langle \iota^{i,\beta}\mid i<\mu\rangle$ is strictly increasing, 
it follows that $C_\beta$ covers the intersection of the clubs $\{j<\mu\mid g_1[j]\s j\}$ and $\{j<\mu\mid g_2[j]\s j\}$.
\end{proof}

Next, by the choice of $\vec f$, fix $j^*<\mu$ and $b\in\mathcal B$ with $\min(b)>\epsilon$ such that $f_{j^*}[b]=\{\gamma\}$.
Then, by the choice of $\psi$, pick $j\in D\cap \bigcap_{\beta\in b}C_\beta$ such that $\psi(j)=j^*$.
In effect, for all $\beta\in b$, $\phi_{f_{\psi(j)}(\beta)}=\phi_\gamma$,
meaning that $$\langle (a^{j,\beta}_\iota,\delta^{j,\beta}_\iota)\mid \iota<\mu\rangle=\langle (a_\iota,\delta_\iota)\mid \iota<\mu\rangle,$$
and in particular, $(\bigcup_{\iota<\mu}a^{j,\beta}_\iota)\s\epsilon=\dom(\pi)$.

Now, let $\beta\in b$; we have:
\begin{enumerate}
\item $\{\iota<\mu\mid \pi[a_\iota]\cap j\neq\emptyset\}\s j$;
\item $\pi [\bigcup_{i<j}a_{\iota^{i,\beta}}^{i,\beta}]\s j$;
\item $\sup_{i<j}(\iota^{i,\beta}+1)=j$.
\end{enumerate}

By Clause~(1), $\pi[a_j]\cap j=\emptyset$. 
Together with Clause~(2),
it thus follows that $a_j\s\epsilon\setminus\bigcup_{i<j}a_{\iota^{i,\beta}}^{i,\beta}\s \beta\setminus\bigcup_{i<j}a_{\iota^{i,\beta}}^{i,\beta}$.
So, by Clause (3) and the definition of $\iota^{j,\beta}$, we infer that $\iota^{j,\beta}=j$.
Altogether, for all $\alpha\in a_j$,
$\min\{i<\mu\mid \alpha\in a_{\iota^{i,\beta}}^{i,\beta}\}=j$,
and hence $c(\alpha,\beta)=\delta_j^{j,\beta}=\delta_j$.

Finally, since $\phi_\gamma\unlhd\phi$ and $\im(\phi)=\{ (a,\delta(a))\mid a\in\mathcal A\}$,
we may find some $a\in\mathcal A$ such that $a_j\supseteq a$ and $\delta_j=\delta(a)$.
Then, $c[a\times b]\s c[a_j\times b]=\{\delta_j\}=\{\delta(a)\}$. This is a contradiction.
\end{proof}

\begin{cor} Suppose that $\mu$ is an infinite regular cardinal and $\stick(\mu^+)$ holds.
Then $\pl_2(\mu^+,E^{\mu^+}_\mu,\mu)$ holds, as well.
\end{cor}
\begin{proof} $\br$ If $\mu=\aleph_0$, then a straight-forward adjustment of Galvin's proof from \cite{galvin}
shows that $\stick(\aleph_1)$ implies $\pr_1(\aleph_1,\aleph_1,\allowbreak\aleph_1,\aleph_0)$.
Now, appeal to Corollary~\ref{cor43}. 

$\br$ If $\mu>\aleph_0$, 
then by Theorem~\ref{thm62}, in particular, $\pr_1(\mu^+,\mu^+,\mu^+,\mu)$ holds.
Now, appeal to Lemma~\ref{Lemma35} and Theorem~\ref{thm42a}.
\end{proof}

\section{From a proxy principle}\label{proxysect}

\begin{thm}\label{lemma51}
Suppose that $\chi\le\kappa$, $\Delta\s\kappa$,
and $\langle h_\delta:C_\delta\rightarrow\kappa\mid \delta<\kappa\rangle$ 
is a sequence satisfying the following:
\begin{enumerate}
\item $\vec C:=\langle C_\delta\mid\delta<\kappa\rangle$ is a $C$-sequence;
\item For every $\eth<\kappa$ and $\delta\in \acc(C_\eth)\cap\Delta$, 
there exists $\epsilon\in C_\delta$ such that $h_\eth\restriction[\epsilon,\delta)=h_\delta\restriction[\epsilon,\delta)$;
\item For every $\sigma<\chi$ and every club $D\s\kappa$, there exists $\delta\in\Delta\cap E^\kappa_{>\sigma}$ such that $\sup(\nacc(C_\delta)\cap D)=\delta$;
\item For every $\sigma<\chi$, every pairwise disjoint subfamily $\mathcal A\s [\kappa]^{\sigma}$
of size $\kappa$
and every $\tau<\kappa$, there exists $\delta\in\Delta\cap\acc(\kappa)$ such that
$$\sup\{\min(x)\mid x\in \mathcal A\cap\mathcal P(C_\delta)\ \&\ h_\delta[x]=\{\tau\}\}=\delta.$$
\end{enumerate}
Then $\pl_1(\kappa,\kappa,\chi)$ holds.
\end{thm}
\begin{proof} We may assume that $C_{\delta+1}=\{\delta\}$ for every $\delta<\kappa$.
We shall now walk along $\vec C$.
Fix a bijection $\pi:\kappa\leftrightarrow\kappa\times\kappa$.
Define a transformation $\mathbf t:[\kappa]^2\rightarrow[\kappa]^3$, letting $\mathbf t(\alpha,\beta):=(\tau,\gamma, \delta)$ 
provided that the following conditions are met:
\begin{itemize}
\item $(\eta,\tau):=\pi(h_{\min(\im(\tr(\alpha,\beta)))}(\alpha))$ and $\max\{\eta+1,\tau\}<\gamma$,
\item $\delta=\Tr(\alpha,\beta)(\eta_{\alpha,\beta})$,
\item $\gamma=\Tr(\eta+1,\alpha)(\eta_{\eta+1,\alpha})$.
\end{itemize}
Otherwise, let $\mathbf t(\alpha,\beta):=(0,\alpha,\beta)$.

We verify that this works. Given $\sigma<\chi$ and 
a pairwise disjoint subfamily $\mathcal{A}\s [\kappa]^{\sigma}$ of size $\kappa$,
we shall find a stationary subset $S\s\kappa$ witnessing the definition of $\pl_1(\kappa,\kappa,\chi)$.

\begin{claim}\label{claim731} There exist a stationary $\Gamma\s\kappa$, a sequence $\langle \mathcal A_\gamma\mid \gamma\in\Gamma\rangle$ and an ordinal $\eta<\kappa$ such that, for all $\gamma\in \Gamma$,
$\mathcal A_\gamma\in[\mathcal A]^\kappa$ and, for all $x\in\mathcal A_\gamma$ and $\alpha\in x$:
\begin{itemize}
\item $\lambda(\gamma,\alpha)=\eta<\gamma<\alpha$;
\item $\eta_{\gamma,\alpha}=\rho_2(\gamma,\alpha)$.
\end{itemize}
\end{claim}
\begin{proof} The proof uses Clause (2) and (3) and is almost identical to that of Lemma~\ref{lemma316}.
\end{proof}

Let $\langle \mathcal A_\gamma\mid\gamma\in\Gamma\rangle$ and $\eta$ be given by the preceding claim.
By Clause~(4), for every $\tau<\kappa$ and $\gamma\in\Gamma$, we may let $\zeta_{\tau,\gamma}$ denote the least $\zeta\in\Delta\cap\acc(\kappa)$
to satisfy:
$$\sup\{\min(x)\mid x\in \mathcal A_\gamma\cap\mathcal P(C_\zeta)\ \&\ h_\zeta[x]=\{\pi^{-1}(\eta,\tau)\}\}=\zeta.$$

Fix a club $E\s\acc(\kappa)$ with the property that, for every $(\tau,\gamma,\delta)\in\kappa\circledast\Gamma\circledast E$,
$\zeta_{\tau,\gamma}<\delta$. 
We claim that $S:=\Gamma\cap E$ is as sought. To see this,
let $(\tau,\gamma, \delta)\in\kappa\circledast S\circledast S$ be arbitrary. 

Let $\zeta:=\zeta_{\tau,\gamma}$, so that $\zeta<\delta$.
Using Clause~(2) and Fact~\ref{last}(2), 
fix $\epsilon\in C_\zeta$ such that $h_{\last{\zeta}{\delta}}\restriction[\epsilon,\zeta)=h_\zeta\restriction[\epsilon,\zeta)$.
Using Fact~\ref{last}(1), pick $a\in \mathcal A_\gamma\cap\mathcal P(C_\zeta)$ with
$\min(a)>\max\{\lambda(\last{\zeta}{\delta},\delta),\epsilon\}$ such that $(\pi\circ h_\zeta)[a]=\{(\eta,\tau)\}$.
Pick $b\in\mathcal A_\delta$ arbitrarily.

\begin{claim} Let $(\alpha,\beta)\in a\times b$. Then $\mathbf t(\alpha,\beta)=(\tau,\gamma,\delta)$.
\end{claim}
\begin{proof} As $b\in\mathcal A_\delta$ and $\beta\in b$, $\lambda(\delta,\beta)=\eta<\gamma<\alpha<\zeta<\delta<\beta$.
So, by Fact~\ref{fact2}, $$\tr(\alpha,\beta)=\tr(\delta,\beta){}^\smallfrown\tr(\alpha,\delta).$$
It thus follows from $\eta_{\delta,\beta}=\rho_2(\delta,\beta)$
that $\Tr(\alpha,\beta)(\eta_{\alpha,\beta})=\delta$.

Next, since $\lambda(\last{\zeta}{\delta},\delta)<\min(a)\le\alpha<\zeta<\delta$, we have
$$\tr(\alpha,\delta)=\tr(\last{\zeta}{\delta},\delta){}^\smallfrown\tr(\alpha,\last{\zeta}{\delta}).$$
As $\alpha\in C_\zeta\cap[\epsilon,\zeta)$ and $\pi(h_\zeta(\alpha))=(\eta,\tau)$,
we infer that $\alpha\in C_{\last{\zeta}{\delta}}$ and $\pi(h_{\last{\zeta}{\delta}}(\alpha))=(\eta,\tau)$.
Altogether, $\min(\im(\tr(\alpha,\beta)))=\min(\im(\alpha,\delta)))=\last{\zeta}{\delta}$,
and $$\pi(h_{\min(\im(\tr(\alpha,\beta)))}(\alpha))=(\eta,\tau).$$

Finally, since $\lambda(\gamma,\alpha)=\eta<\eta+1<\gamma<\alpha$,
$\tr(\eta+1,\alpha)=\tr(\gamma,\alpha){}^\smallfrown\tr(\eta+1,\gamma)$,
and as $\eta_{\gamma,\alpha}=\rho_2(\gamma,\alpha)$,
we infer that $\Tr(\eta+1,\alpha)(\eta_{\eta+1,\alpha})=\gamma$.
\end{proof}
This completes the proof.
\end{proof}

\begin{thm}\label{pr51} Suppose that $\kappa=\mu^+$ for some infinite regular cardinal $\mu$ and $\p^-(\kappa,\kappa^+,{\sq^*},1,\{E^\kappa_\mu\},2)$ holds.
 Then $\pl_2(\kappa,\kappa,\mu)$ holds, as well.
\end{thm}
\begin{proof}  By Lemma~\ref{Lemma35}, Theorem~\ref{thm47}(2)
and Lemma~\ref{lemma316}, it suffices to prove that $\pr_1(\kappa,\kappa,\kappa,\mu)$ holds.
We shall establish that $\pl_1(\kappa,\kappa,\mu)$ holds, using Theorem~\ref{lemma51}.

For every nonzero $\delta<\kappa$, fix a surjection $\psi_\delta:\mu\rightarrow\delta$.
Then, let 
\begin{itemize}
\item $\mathcal X:=[\kappa]^{<\omega}\cup\{ \psi_\delta[\eta]\setminus\alpha\mid \alpha<\delta<\kappa, \eta<\mu\}$, and
\item $\mathcal F:=\{ \cl(x)\times\{j\}\mid x\in\mathcal X, j<\kappa\}$.
\end{itemize}

Fix an enumeration (possibly, with repetitions) $\langle f_\gamma\mid \gamma<\kappa\rangle$ of $\mathcal F$ 
such that, for every $\gamma<\kappa$, $\dom(f_\gamma)\s\gamma$.
For every $(\beta,\gamma)\in[\kappa]^2$, set $f_\gamma^\beta:=f_\gamma\restriction(\beta,\gamma)$,
so that $f_\gamma^\beta$ 
is a constant function, and $\dom(f_\gamma^\beta)$ is a closed set of ordinals of order-type $<\mu$.

Let $\vec D=\langle D_\delta\mid\delta<\kappa\rangle$ and $\Delta\s E^\kappa_\mu$ be witnesses to  $\p^-(\kappa,\kappa^+,{\sq^*},1,\{E^\kappa_\mu\},2)$.
We now construct a sequence $\langle h_\delta:C_\delta\rightarrow\kappa\mid\delta<\kappa\rangle$ 
that satisfy the requirements of Theorem~\ref{lemma51}, with $\chi:=\mu$.

For each $\delta\in E^\kappa_{<\mu}$, fix a closed subset $C_\delta\s\delta$ with $\sup(C_\delta)=\sup(\delta)$ and $\otp(C_\delta)=\cf(\delta)$,
and then let $h_\delta:=C_\delta\times\{0\}$.
Next, for each $\delta\in E^\kappa_\mu$, let
$$h_\delta:=(D_\delta\times\{0\})\cup\bigcup\{f_{\gamma}^\beta\mid \beta\in D_\delta, \gamma=\min(D_{\delta}\setminus(\beta+1))\},$$
so that $C_\delta:=\dom(h_\delta)$ is a club in $\delta$ and $\acc(C_\delta)\cap E^\kappa_\mu=\acc(D_\delta)\cap E^\kappa_\mu$.
\begin{claim} Suppose $\delta<\kappa$ and $\bar\delta\in\acc(C_\delta)\cap\Delta$. Then there exists $\epsilon\in C_{\bar\delta}$ such that $h_\delta\restriction[\epsilon,\bar\delta)=h_{\bar\delta}\restriction[\epsilon,\bar\delta)$.
\end{claim}
\begin{proof} Clearly, $\delta\in\Delta$ and $\bar\delta\in\acc(D_\delta)$. In particular, $\varepsilon:=\sup((D_\delta\cap\bar\delta)\symdiff D_{\bar\delta})$ is $<\bar\delta$.
Let $\epsilon:=\min(D_{\bar\delta}\setminus(\varepsilon+1))$.
Then $D_\delta\cap[\epsilon,\bar\delta)=D_{\bar\delta}\cap[\epsilon,\bar\delta)$,
and it follows that $h_\delta\restriction[\epsilon,\bar\delta)=h_{\bar\delta}\restriction[\epsilon,\bar\delta)$.
\end{proof}

\begin{claim}\label{claim532} Let $A\s\kappa$ be cofinal. Then there exists $\delta\in\Delta$ such that $\sup(\nacc(C_\delta)\cap A)=\delta$.
\end{claim}
\begin{proof} As $\{ \{\alpha \}\mid \alpha \in A\}\s[\kappa]^{<\omega}\s\mathcal X$,
we infer that $\{ \{(\alpha,0)\}\mid \alpha \in A\}\s\mathcal F$. So, we may recursively construct a sequence $\langle (\alpha_i,\beta_i)\mid i<\kappa\rangle$ such that, for all $(i,i')\in[\kappa]^2$:
\begin{itemize}
\item $\alpha_i\in A$;
\item $f_{\beta_i}=\{(\alpha_i,0)\}$;
\item $\beta_i<\alpha_{i'}$.
\end{itemize}

Set $B:=\{\beta_i\mid i<\kappa\}$.
Now, by the choice of $\vec D$, we may fix $\delta\in \Delta$ for which the following set is cofinal in $\delta$:
$$\Gamma:=\{\gamma\in \nacc(D_\delta)\cap B\mid \exists \varepsilon\in B(\sup(D_\delta\cap\gamma)\le\varepsilon<\gamma)\}.$$
Now, given $\gamma\in\Gamma$, fix $\varepsilon_{\gamma}\in B$ such that $\sup(D_\delta\cap\gamma)\le\varepsilon_{\gamma}<\gamma$.
Find $(i,i')\in[\kappa]^2$ such that $\varepsilon_\gamma=\beta_i$ and $\gamma=\beta_{i'}$. Set $\beta:=\sup(D_\delta\cap\gamma)$.
Then $\beta\le\beta_i<\alpha_{i'}$ and $\beta_{i'}=\gamma$, so that $f_\gamma^\beta=f_{\gamma}=\{(\alpha_{i'},0)\}$
and $e_\delta\cap(\beta,\gamma)=\{\alpha_{i'}\}$.
Thus, we have established that, for every $\gamma\in\Gamma$, $\nacc(C_\delta)\cap A\setminus\varepsilon_\gamma$ is nonempty.
Consequently, $\sup(\nacc(C_\delta)\cap A)=\delta$.
\end{proof}

\begin{claim} Suppose $\mathcal A\s [\kappa]^{<\mu}$ is a family
consisting of $\kappa$ many pairwise disjoint sets, 
and $\tau<\kappa$.
Then there exist $\delta\in\Delta$ such that
$$\sup\{\min(x)\mid x\in \mathcal A\cap\mathcal P(C_\delta)\ \&\ h_\delta[x]=\{\tau\}\}=\delta.$$
\end{claim}
\begin{proof}  For every $x\in\mathcal A$, 
we may find a large enough $\delta<\kappa$ such that $x\in[\delta]^{<\mu}$,
and so, by regularity of $\mu$, we may find $\eta<\mu$ such that $x\s\psi_\delta[\eta]$
so that $x':=\psi_\delta[\eta]\setminus\min(x)$ is an element of $\mathcal X$ satisfying $x\s x' $ and $\min(x)=\min(x')$.
It follows that we may recursively construct a sequence $\langle (x_i,\beta_i)\mid i<\kappa\rangle$ such that, for all $(i,i')\in[\kappa]^2$:
\begin{itemize}
\item $x_i\in\mathcal A$;
\item $(x_i\times\{\tau\})\s f_{\beta_i}$;
\item $\beta_i<\min(\dom(f_{\beta_{i'}}))$.
\end{itemize}
Set $B:=\{\beta_i\mid i<\kappa\}$.

Now, by the choice of $\vec D$, we may fix $\delta\in\Delta\cap\acc(\kappa)$ for which the following set is cofinal in $\delta$:
$$\Gamma:=\{\gamma\in \nacc(D_\delta)\cap B\mid \exists \varepsilon\in B(\sup(D_\delta\cap\gamma)\le\varepsilon<\gamma)\}.$$

Let $\gamma\in\Gamma$. Fix $\varepsilon_{\gamma}\in B$ such that $\sup(D_\delta\cap\gamma)\le\varepsilon_{\gamma}<\gamma$.
Find $(i,i')\in[\kappa]^2$ such that $\varepsilon_\gamma=\beta_i$ and $\gamma=\beta_{i'}$. Set $\beta:=\sup(D_\delta\cap\gamma)$.
Then $\beta\le\beta_i<\min(\dom(f_{\beta_{i'}}))$ and $\beta_{i'}=\gamma$, so that 
$h_\delta\restriction (\beta,\gamma)=f_\gamma^\beta=f_{\gamma}\supseteq(x_{i'}\times\{\tau\})$.
As $x_{i'}\in\mathcal A$ with $\varepsilon_\gamma<\min(x_{i'})$ and $\sup\{\varepsilon_\gamma\mid \gamma\in\Gamma\}=\delta$, we are done.
\end{proof}
Now, we are in a position to appeal to Theorem~\ref{lemma51}.
\end{proof}

The next corollary yields Clause~(2) of Theorem~A.

\begin{cor}\label{cor73} For every infinite regular cardinal $\mu$, any of the following imply that
$\pl_2(\mu^+,\mu^+,\mu)$ holds:
\begin{enumerate}
\item $\clubsuit(E^{\mu^+}_\mu)$ holds;
\item  $(\mu^+)^{\aleph_0}=\mu^+$ and $\clubsuit(S)$ holds for some nonreflecting stationary $S\s\mu^+$.
\end{enumerate}
\end{cor}
\begin{proof} By Theorem~\ref{pr51},
it suffices to prove that $\p^-(\mu^+,\mu^{++},{\sq^*},1,\{E^{\mu^+}_\mu\},2)$ holds.
It is clear that the hypothesis of Clause~(1) implies this instance.
By \cite[Lemma~4.20]{paper23}, 
also the hypothesis of Clause~(2) implies this instance.
\end{proof}

\begin{cor} For every infinite regular cardinal $\mu$, $V^{\add(\mu,1)}\models \pl_2(\mu^+,\mu^+,\mu)$.\footnote{Here, $\mu^+$ stands for the successor of $\mu$ in the generic extension.}
\end{cor}
\begin{proof} As $\add(\mu,1)$ is equivalent to $\add(\mu,2)$,
we may assume that we are forcing over a model of $\mu^{<\mu}=\mu$.
Now, by the same proof of \cite[Theorem~2.3]{paper12} (cf.~\cite[Theorem~4.2]{paper22}), 
while ignoring any aspect of coherence (as it is not needed here; just ensuring that all the clubs have order-type at most $\mu$ is enough),
$V^{\add(\mu,1)}\models \p^-(\kappa,\kappa^+,{\sq^*},1,\{E^\kappa_\mu\},2)$ for $\kappa:=\mu^+$.
Finally, appeal to Theorem~\ref{pr51}.
\end{proof}

\begin{cor} Suppose that $\mu$ is a regular uncountable cardinal
satisfying $2^\mu=\mu^+$, and $\mathbb P$ is a $\mu^+$-cc notion of forcing of size $\le\mu^+$ that preserves the regularity of $\mu$ but does not satisfy the ${}^\mu\mu$-bounding property.
Then  $V^{\mathbb P}\models\pl_2(\mu^+,\mu^+,\mu)$.
\end{cor}
\begin{proof} By \cite[Theorem~3.4]{paper26}, if we also assume that $\mu^{<\mu}=\mu$,
then, in $V^{\mathbb P}$, $\p^*(E^{\mu^+}_\mu,\mu)$ holds,
so that, in particular, 
$\p^-(\mu^+,\mu^+,{\sq},1,\{E^{\mu^+}_\mu\},2,{<}\infty,\mathcal E_\mu)$ holds.
By waiving the hypothesis ``$\mu^{<\mu}=\mu$'', the only thing that breaks down is \cite[Claim~3.4.3]{paper26},
meaning that $\p^-(\mu^+,\mu^{++},{\sq},\mu^+,\{E^{\mu^+}_\mu\},2,{<}\infty,\mathcal E_\mu)$ holds in $V^{\mathbb P}$, instead.
In particular, $V^{\mathbb P}\models\p^-(\kappa,\kappa^+,{\sq^*},1,\{E^\kappa_\mu\},2)$ holds for $\kappa:=\mu^+$,
and we may appeal to Theorem~\ref{pr51}.
\end{proof}

\begin{thm}\label{thm58} Suppose that $\kappa=\kappa^{<\kappa}$ is a Mahlo cardinal and $\p^-(\kappa,\kappa^+,{\sq^*},1,\allowbreak\{\reg(\kappa)\},2)$ holds.
 Then $\pl_2(\kappa,\kappa,\kappa)$ holds, as well.
\end{thm}
\begin{proof} 
By Lemma~\ref{Lemma35}, Theorem~\ref{thm47}(2)
and Lemma~\ref{lemma316}, it suffices to prove that $\pr_1(\kappa,\kappa,\kappa,\kappa)$ holds.
We shall moreover prove that $\pl_1(\kappa,\kappa,\kappa)$ holds, using Theorem~\ref{lemma51}.

Let $\mathcal F:=\bigcup\{ {}^{\cl(x)}\kappa\mid x\in[\kappa]^{<\kappa}\}$.
Fix an enumeration (possibly, with repetitions) $\langle f_\gamma\mid \gamma<\kappa\rangle$ of $\mathcal F$ 
such that, for every $\gamma<\kappa$, $\dom(f_\gamma)\s\gamma$.
For every $(\beta,\gamma)\in[\kappa]^2$, set $f_\gamma^\beta:=f_\gamma\restriction(\beta,\gamma)$,
so that $f_\gamma^\beta$ 
is a constant function, and $\dom(f_\gamma^\beta)$ is a closed set of ordinals of order-type $<\kappa$.

Let $\vec D=\langle D_\delta\mid\delta<\kappa\rangle$ and $\Delta\s\reg(\kappa)$ be witnesses to the fact  $\p^-(\kappa,\kappa^+,{\sq^*},1,\allowbreak\{\reg(\kappa)\},2)$ holds.
We now construct a sequence $\langle h_\delta:C_\delta\rightarrow\kappa\mid\delta<\kappa\rangle$ 
that satisfy the requirements of Theorem~\ref{lemma51}, with $\chi:=\kappa$.

For each $\delta\in\acc(\kappa)$, let
$$h_\delta:=(D_\delta\times\{0\})\cup\bigcup\{f_{\gamma}^\beta\mid \beta\in D_\delta, \gamma=\min(D_{\delta}\setminus(\beta+1)), \otp(\dom(f_\gamma^\beta))<\beta\},$$
so that $C_\delta:=\dom(h_\delta)$ is a club in $\delta$ and $\acc(C_\delta)\cap\reg(\kappa)=\acc(D_\delta)\cap\reg(\kappa)$.
\begin{claim} Suppose $\delta<\kappa$ and $\bar\delta\in\acc(C_\delta)\cap\Delta$. Then there exists $\epsilon\in C_{\bar\delta}$ such that $h_\delta\restriction[\epsilon,\bar\delta)=h_{\bar\delta}\restriction[\epsilon,\bar\delta)$.
\end{claim}
\begin{proof} By the choice of $\vec D$, $\varepsilon:=\sup((D_\delta\cap\bar\delta)\symdiff D_{\bar\delta})$ is $<\bar\delta$.
Let $\epsilon:=\min(D_{\bar\delta}\setminus(\varepsilon+1))$.
Then $D_\delta\cap[\epsilon,\bar\delta)=D_{\bar\delta}\cap[\epsilon,\bar\delta)$,
and it follows that $h_\delta\restriction[\epsilon,\bar\delta)=h_{\bar\delta}\restriction[\epsilon,\bar\delta)$.
\end{proof}

As made clear by the proof of Claim~\ref{claim532}, for every cofinal $A\s\kappa$, 
there exists $\delta\in\Delta$ such that $\sup(\nacc(C_\delta)\cap A)=\delta$.

\begin{claim} Suppose $\sigma<\kappa$ and $\mathcal A\s [\kappa]^{\sigma}$ is a family
consisting of $\kappa$ many pairwise disjoint sets, 
and $\tau<\kappa$.
Then there exist $\delta\in\Delta$ such that
$$\sup\{\min(x)\mid x\in \mathcal A\cap\mathcal P(C_\delta)\ \&\ h_\delta[x]=\{\tau\}\}=\delta.$$
\end{claim}
\begin{proof}  Recursively construct a sequence $\langle (x_i,\beta_i)\mid i<\kappa\rangle$ such that, for all $(i,i')\in[\kappa]^2$:
\begin{itemize}
\item $x_i\in\mathcal A$;
\item $f_{\beta_i}$ is the constant function from $x_i$ to $\tau$;
\item $\beta_i<\min(x_{i'})$.
\end{itemize}
Set $B:=\{\beta_i\mid i<\kappa\}$.
By the choice of $\vec D$, we may fix $\delta\in\Delta\cap\acc(\kappa\setminus\sigma)$ for which the following set is cofinal in $\delta$:
$$\Gamma:=\{\gamma\in \nacc(D_\delta)\cap B\mid \exists \varepsilon\in B(\sup(D_\delta\cap\gamma)\le\varepsilon<\gamma)\}.$$

Let $\gamma\in\Gamma$ with $\sup(D_\delta\cap\gamma)\ge\sigma$. Fix $\varepsilon_{\gamma}\in B$ such that $\sup(D_\delta\cap\gamma)\le\varepsilon_{\gamma}<\gamma$.
Find $(i,i')\in[\kappa]^2$ such that $\varepsilon_\gamma=\beta_i$ and $\gamma=\beta_{i'}$. Set $\beta:=\sup(D_\delta\cap\gamma)$.
Then $$\sigma\le\beta\le\beta_i<\min(x_{i'})=\min(\dom(f_{\beta_{i'}}))$$ and $\beta_{i'}=\gamma$, so that 
$h_\delta\restriction (\beta,\gamma)=f_\gamma^\beta=f_{\gamma}=(x_{i'}\times\{\tau\})$.
As $x_{i'}\in\mathcal A$ with $\varepsilon_\gamma<\min(x_{i'})$ and $\sup\{\varepsilon_\gamma\mid \gamma\in\Gamma\ \&\ \sup(D_\delta\cap\gamma)\ge\sigma\}=\delta$, we are done.
\end{proof}
Now, we are in a position to appeal to Theorem~\ref{lemma51}.
\end{proof}

The next corollary yields Clause~(4) of Theorem~B.

\begin{cor}\label{cor77} Suppose that $\kappa$ is a Mahlo cardinal,
and there exists a nonreflecting stationary $E\s \kappa$ such that $\diamondsuit(E)$ holds.
If $\square(E)$ holds
 or if there exists a nonreflecting stationary subset of $\reg(\kappa)$, 
then $\pl_2(\kappa,\kappa,\kappa)$ holds.
\end{cor}
\begin{proof} Recall that $\diamondsuit(E)$ implies $\kappa^{<\kappa}=\kappa$.
So, by Theorem~\ref{thm58},
it suffices to prove that $\p^-(\kappa,\kappa^+,{\sq^*},1,\{S\},2)$ holds for some stationary subset $S$ of $\reg(\kappa)$.

$\br$  If $\square(E)$ holds, then by \cite[Corollary~4.19(2)]{paper23},
$\p^-(\kappa,2,{\sq^*},1,\{S\},2)$ holds for every stationary $S\s\kappa$.

$\br$ Suppose that $S$ is a nonreflecting stationary subset of $\reg(\kappa)$.
By \cite[Corollary~4.27]{paper23}, if in addition $\kappa$ is a strong limit, then $\p^-(\kappa,\kappa	,{\sq^*},1,\{S\},2)$ holds.
The same proof shows that, in the general case, $\p^-(\kappa,\kappa^+,{\sq^*},1,\allowbreak\{S\},2)$ holds.
\end{proof}

We can now derive Theorem~D.
\begin{cor} Suppose that $V=L$. For every regular uncountable cardinal $\kappa$ and every regular cardinal $\chi\le\chi(\kappa)$,
$\pl_2(\kappa,\kappa,\chi)$ holds.
\end{cor}
\begin{proof} There are four cases to consider:

$\br$ If $\kappa=\mu^+$ for $\mu$ regular, 
then by \cite[Lemma~2.2(5)]{paper35}, $\chi(\kappa)=\mu$,
and by \cite{jensen72}, $\diamondsuit(E^{\mu^+}_\mu)$ holds, 
so by Corollary~\ref{cor73}(1), $\pl_2(\kappa,\kappa,\mu)$ holds.

$\br$ If $\kappa=\mu^+$ for $\mu$ singular, 
then by \cite[Lemma~2.2(5)]{paper35}, $\chi(\kappa)=\mu$,
and by \cite{jensen72}, for every regular $\chi\le\mu$, there exists a nonreflecting stationary subset of $E^\kappa_\chi$,
so by the main result of \cite{paper15}, $\pr_1(\kappa,\kappa,\kappa,\chi)$ holds.
In addition, by \cite[Corollary~1.10(5)]{paper22}, $\boxtimes^-(\kappa)$ holds,
so by Corollary~\ref{cor48}(2), $\pl_2(\kappa,\kappa,\chi)$ holds.

$\br$ If $\kappa$ is inaccessible which is not weakly compact, then by \cite[Lemma~2.2(5)]{paper35}, $\chi(\kappa)=\kappa$, 
and by \cite{jensen72}, there exists a nonreflecting stationary subset $E\s\kappa$ such $\diamondsuit(E)$ and $\square(E)$ both hold.
So, by Corollary~\ref{cor77}, $\pl_2(\kappa,\kappa,\kappa)$ holds.

$\br$ If $\kappa$ is weakly compact, then $\chi(\kappa)=0$, so there is nothing to prove here.
\end{proof}

\section*{Acknowledgments}
The main results of this paper were presented by the first author
at an online meeting of the Toronto Set Theory Seminar, February 2021.
He thanks the organizers for the invitation and the participants for their feedback.

The first author is partially supported by the European Research Council (grant agreement ERC-2018-StG 802756) and by the Israel Science Foundation (grant agreement 2066/18).
The second author is supported by the Foreign Postdoctoral Fellowship Program of the Israel Academy of Sciences and Humanities and by the Israel Science Foundation (grant agreement 2066/18).


\begin{thebibliography}{BGKT78}

\bibitem[BGKT78]{MR511564}
S.~Broverman, J.~Ginsburg, K.~Kunen, and F.~D. Tall.
\newblock Topologies determined by {$\sigma $}-ideals on {$\omega _{1}$}.
\newblock {\em Canadian J. Math.}, 30(6):1306--1312, 1978.

\bibitem[BR17]{paper22}
Ari~Meir Brodsky and Assaf Rinot.
\newblock A microscopic approach to {S}ouslin-tree constructions. {P}art {I}.
\newblock {\em Ann. Pure Appl. Logic}, 168(11):1949--2007, 2017.

\bibitem[BR19a]{paper29}
Ari~Meir Brodsky and Assaf Rinot.
\newblock Distributive {A}ronszajn trees.
\newblock {\em Fund. Math.}, 245(3):217--291, 2019.

\bibitem[BR19b]{paper26}
Ari~Meir Brodsky and Assaf Rinot.
\newblock More notions of forcing add a {S}ouslin tree.
\newblock {\em Notre Dame J. Form. Log.}, 60(3):437--455, 2019.

\bibitem[BR21]{paper23}
Ari~Meir Brodsky and Assaf Rinot.
\newblock A microscopic approach to {S}ouslin-tree constructions. {P}art {II}.
\newblock {\em Ann. Pure Appl. Logic}, page 102904, 2021.

\bibitem[Gal80]{galvin}
Fred Galvin.
\newblock Chain conditions and products.
\newblock {\em Fund. Math.}, 108(1):33--48, 1980.

\bibitem[IR21]{paper47}
Tanmay Inamdar and Assaf Rinot.
\newblock Was {U}lam right?
\newblock \verb"http://assafrinot.com/paper/47", 2021.
\newblock In preparation.

\bibitem[Jen72]{jensen72}
R.~Bj{\"o}rn Jensen.
\newblock The fine structure of the constructible hierarchy.
\newblock {\em Ann. Math. Logic}, 4:229--308; erratum, ibid. 4 (1972), 443,
  1972.
\newblock With a section by Jack Silver.

\bibitem[KRS21]{paper49}
Menachem Kojman, Assaf Rinot, and Juris Steprans.
\newblock Advances on strong colorings over partitions.
\newblock \verb"http://assafrinot.com/paper/49", 2021.
\newblock Submitted April 2021.

\bibitem[LHR18]{paper34}
Chris Lambie-Hanson and Assaf Rinot.
\newblock Knaster and friends {I}: {C}losed colorings and precalibers.
\newblock {\em Algebra Universalis}, 79(4):Art. 90, 39, 2018.

\bibitem[LHR21]{paper35}
Chris Lambie-Hanson and Assaf Rinot.
\newblock Knaster and friends {II}: {T}he {C}-sequence number.
\newblock {\em J. Math. Log.}, 21(1):2150002, 54, 2021.

\bibitem[PW18]{MR3742590}
Yinhe Peng and Liuzhen Wu.
\newblock A {L}indel\"{o}f group with non-{L}indel\"{o}f square.
\newblock {\em Adv. Math.}, 325:215--242, 2018.

\bibitem[Rin12]{paper13}
Assaf Rinot.
\newblock Transforming rectangles into squares, with applications to strong
  colorings.
\newblock {\em Adv. Math.}, 231(2):1085--1099, 2012.

\bibitem[Rin14]{paper15}
Assaf Rinot.
\newblock Complicated colorings.
\newblock {\em Math. Res. Lett.}, 21(6):1367--1388, 2014.

\bibitem[Rin15]{paper12}
Assaf Rinot.
\newblock Chromatic numbers of graphs - large gaps.
\newblock {\em Combinatorica}, 35(2):215--233, 2015.

\bibitem[Rin17]{paper24}
Assaf Rinot.
\newblock Higher {S}ouslin trees and the {GCH}, revisited.
\newblock {\em Adv. Math.}, 311(C):510--531, 2017.

\bibitem[RZ21]{paper44}
Assaf Rinot and Jing Zhang.
\newblock Transformations of the transfinite plane.
\newblock {\em Forum Math. Sigma}, 9(e16):1--25, 2021.

\bibitem[She88]{shelah_productivity}
Saharon Shelah.
\newblock Successors of singulars, cofinalities of reduced products of
  cardinals and productivity of chain conditions.
\newblock {\em Israel J. Math.}, 62(2):213--256, 1988.

\bibitem[She97]{Sh:572}
Saharon Shelah.
\newblock Colouring and non-productivity of $\aleph_2$-cc.
\newblock {\em Annals of Pure and Applied Logic}, 84:153--174, 1997.

\bibitem[Tod07]{TodWalks}
Stevo Todorcevic.
\newblock {\em Walks on ordinals and their characteristics}, volume 263 of {\em
  Progress in Mathematics}.
\newblock Birkh\"auser Verlag, Basel, 2007.

\end{thebibliography}
\end{document}